\documentclass[12pt]{amsart}
\usepackage{graphicx}
\usepackage{amsfonts}
\usepackage{amsmath}
\usepackage{cases}
\usepackage{amsfonts}
\usepackage{amssymb}
\usepackage[margin=1in]{geometry}
\usepackage{mathrsfs}
\usepackage{verbatim}
\usepackage{hyperref}
\usepackage{datetime}
\usepackage{amstext}
\usepackage{amsxtra}
\usepackage{amsopn}
\usepackage{amsthm}
\newtheorem{thm}{Theorem}[section]
\theoremstyle{definition}
\newtheorem{defn}{Definition}[section]
\theoremstyle{plain}

\theoremstyle{remark}
\newtheorem{rem}{Remark}[section]
\theoremstyle{plain}
\newtheorem{lem}[thm]{Lemma}
\theoremstyle{plain}
\newtheorem{cor}[thm]{Corollary}
\theoremstyle{plain}
\newtheorem{prop}[thm]{Proposition}
\theoremstyle{conjecture}

\usepackage{cancel}
\usepackage{tikz-cd}

\input xy
\xyoption{all}

\newcommand{\R}{\mathbb{R}}
\newcommand{\J}{\mathbb{J}}
\newcommand{\M}{\mathbb{M}}
\newcommand{\X}{\mathbb{X}}
\newcommand{\C}{\mathbb{C}}
\newcommand{\Z}{\mathbb{Z}}
\newcommand{\Q}{\mathbb{Q}}

\newcommand{\norm}[1]{\left\lVert#1\right\rVert}

\begin{document}

\title{On the global moduli of Calabi-Yau threefolds}
\author{Ron Donagi}
\address{Department of Mathematics\\David Rittenhouse Lab.\\
University of Pennsylvania\\
209 South 33rd Street\\Philadelphia, PA  19104-6395}
\email{donagi@math.upenn.edu}

\author{Mark Macerato}
\email{macerato@sas.upenn.edu}

\author{Eric Sharpe}
\address{Department of Physics MC 0435\\
850 Campus Drive\\Virginia Tech\\
Blacksburg, VA  24061}
\email{ersharpe@vt.edu}

\date{\today}

\begin{abstract}
In this note we initiate a program to obtain global descriptions of
Calabi-Yau moduli spaces, to calculate their Picard group, and to identify
 within that group the Hodge line bundle. 
We do this here for several Calabi-Yau's obtained in \cite{DW} as 
crepant resolutions 
of the orbifold quotient of the product of three elliptic curves.
In particular we verify in these cases a recent claim of \cite{KS} by
noting that a power of the Hodge line bundle is trivial -- 
even though in most of these cases the Picard group is infinite.
\end{abstract}

\maketitle

\tableofcontents

\section{Introduction}

Moduli spaces $M_X$ of (complex structures on) a compact Calabi-Yau manifold 
$X$
are central to superstring compactifications, mirror symmetry, conformal
field theory, and numerous other branches of both geometry and physics. They
have the familiar complexity of moduli spaces; in particular, there is a
stacky version as well as an underlying coarse moduli space. The stack is smooth. It is
often easier to understand, and we will deal primarily with it.

While a lot is known about the local structure of these moduli spaces, there
are surprisingly few examples
where the global geometry is fully understood, and much of this is for
moduli spaces of low complex dimensions, one or two.

Of particular interest is an understanding of the Hodge line bundle $\lambda$ on
$M_X$, whose fiber above the (isomorphism class of) a particular $X$ is the line
$H^0(X,\omega_X)$ of top holomorphic forms on $X$. Some startling predictions
have appeared in recent physics literature \cite{KS}, 
to the effect that 
Calabi-Yau moduli spaces $M_X$ always admit a (globally defined)
K\"ahler potential.
Our results verify these predictions in all cases that we consider.
In fact, we show that the Hodge line bundle is not trivial, 
but some finite power of it is trivial. 
This implies the existence of a global K\"ahler potential.

Over $M_X$ there is a universal family $\pi_X: \X \to M_X$ whose fiber above
the isomorphism class $[X]$ of some $X$ is that $X$ itself. The intermediate
Jacobians $J(X)$ fit into a family $\J_X \to M_X$. In \cite{DM} it is shown
that the pull back $\widetilde{\J_X} \to \widetilde{\M_X}$ is an analytically
completely integrable system. Here $\widetilde{\M_X}$ is the space of pairs
$(X,\alpha)$ where $X$ is a complex structure on a Calabi-Yau 
of a given deformation
class, and $\alpha$ is a holomorphic volume form on $X$. In other words,
$\widetilde{\M_X}$ is the complement in $L_X$ of the zero section. This system
is integrable analytically but not algebraically: the fibers are polarized
complex tori, but the polarization is not positive definite, so the fibers
are not abelian varieties. (Instead, the polarization is Lorentzian, with $h^{2,1}$ positive
directions and a single negative direction corresponding to $h^{3,0}$.)
Physically, this corresponds to the fact that $\widetilde{\J_X}$ approximates,
in a large limit, the Ramond moduli space of the theory. 
The new prediction implies that
$\widetilde{\M_X}$ is just the product $M_X \times {\mathbb C}^*$
(up to a finite cover or a more general covering with constant transition functions), 
and the symplectic form on the integrable system $\widetilde{\J_X}$ over it 
splits globally into horizontal and vertical summands.
(When $X$ is not necessarily compact, 
it is possible for the $\J_X$ to be abelian varieties; this happens
if the negative direction, i.e. a 
holomorphic volume form, coincides with a vanishing cycle.
In fact \cite{DDP}, for each Riemann surface $C$ and each ADE group $G$ it is
possible to construct a family of non-compact Calabi-Yau threefolds whose 
${\J_X}$ recovers the Hitchin system $H_{C,G}$. 
These results have been extended recently to include 
the non-simply laced groups BCFG \cite{Beck}.)

It therefore seems worthwhile to construct some global moduli spaces $M_X$ and
to test the physics predictions for them. In the case that $X$ is an elliptic
curve, the results are well known to mathematicians, and have been
summarized for physicists in \cite{S}. The purpose of this work is to
describe some global moduli spaces $M_X$ and to determine their Hodge bundles
$\lambda$ for some genuine, three-dimensional compact
Calabi-Yaus. The moduli spaces we
consider  themselves will also be three dimensional.

The examples we consider are 
crepant resolutions $\overline{X}$ of
orbifolds \[
X:= Y/G
\]
of the product $Y:= E_1\times E_2\times E_3$
of three elliptic curves by the action of a finite group $G$. The latter
contains the subgroup 
\[
G_S \subset  G
\]
 of its `shifts' or translations, 
with a  quotient group 
\[
G_T:=G/G_S.
\]
Elements of $G_T$ are called twists. 
We consider the case where $G_T$ is isomorphic to 
${\mathbb Z}/2 \times {\mathbb Z}/2$,
acting (up to translations) by nontrivial elliptic involutions (sign
changes) on an even number of the $E_i$. It turns out \cite{DW} that all
such groups $G$, and orbifolds $X$, can be described explicitly.

In that work, a particular class of such group actions was designated
`essential'.
It was shown that any orbifold $X$ of this type is isomorphic to one whose
group is essential.
Essential groups were shown to be abelian,
isomorphic to the product $G_S \times G_T$ of their shift and twist parts,
and all their non-trivial elements are of order 2.
So, essentially, the shift group $G_S$ must be a subgroup of the group of
points of order 2 in $Y$, acting by translations. It is therefore isomorphic
to $({\mathbb Z}/2)^r$ for some $r$, $0 \leq r \leq 6$. This $r$ is called
the rank of $G$.
The full essential group $G$ is then isomorphic to $({\mathbb Z}/2)^{r+2}$.

All such orbifolds have been classified in \cite{DW}. 
They fall into 35 types. 
Most of these are known to be in distinct families, but as explained in \cite{DW} and reviewed in our appendix,
there are a few gaps where the possible existence of isomorphisms is still not known. In fact, one pair from the 
list in \cite{DW} has since been shown to be isomorphic, cf. \cite{FRTV13}.
After explaining the required notation, we recall that classification and comment on its current state in the Appendix, in Table~\ref{table:summary-dw}.

In this work we focus on ten of the families in this table, the ones whose moduli spaces are three
dimensional. 
For each of these there is a canonical identification:
\begin{equation}\label{product}
H^3(X,\Q) \cong \otimes_{1=1,2,3}H^1(E_{i},\Q).
\end{equation}
The best known example, denoted (0-1) in \cite{DW}, was
originally studied by \cite{VW95}. In that case the group action has 48
fixed points, leading to singularities of $X$ that need to be resolved (by a
crepant resolution). There are four cases where the group action is free,
leading to smooth quotients of Hodge numbers $(3,3)$. In \cite{DW} these are
denoted (0-4), (1-5), (1-11), and (2-12). (The notation ($r$-$i$) means that the
group $G_S$ has rank $r$, and that this is the $i$-th case listed in 
\cite{DW} with
that given $r$.) There are six further cases,
including the \cite{VW95} orbifold (0-1),
 where the group does have some
fixed points but the number of complex moduli happens to still be $h^{2,1}=3$. 
We tabulate
these ten orbifolds with $h^{2,1} = 3$, giving their symbol from [DW09], their
Hodge numbers,
whatever alternative descriptions are available, and some references where
they are either analyzed or used:
\begin{center}
\begin{tabular}{ccl}
(0-1) & (51,3) & The basic Vafa-Witten orbifold \cite{VW95}, also a
Borcea-Voisin \\
& & BV(18, 4, 0) \cite{Bor97,Voi93}. \\
(0-4) & (3,3) & Occurs in \cite{Ig,Ue,OS, BCDP,T} \\
(1-1) & (27,3) & A $({\mathbb Z}/2)$ free quotient of the basic Vafa-Witten
orbifold \\
(1-5) & (3,3) & A $({\mathbb Z}/2)$ free quotient of (0-4) \cite{DW,T} \\
(1-11) & (3,3) & Another $({\mathbb Z}/2)$ free quotient of (0-4)
\cite{DW,T} \\
(2-1) & (15,3) & A $({\mathbb Z}/2)$ free quotient of (1-1)\\
(2-9) & (27,3) & An orbifold of the $SO(12)$ torus, related to B$_{NAHE+}$
free fermion model \\
& & \cite{Fa92, Fa93, DF04, DW}\\
(2-12) & (3,3) & A $({\mathbb Z}/2)$ free quotient of (1-11) \cite{DW,T} \\
(3-5) & (15,3) & A $({\mathbb Z}/2)$ free quotient of (2-9) \\
(4-1) & (15,3) & Related to `enhanced' B$_{NAHE+}$ free fermion model
\cite{Fa92, Fa93, DF04, DW}
\end{tabular}
\end{center}

Notice that any projective crepant resolution 
of any of these is a genuine $\mathcal{N}=1$ Calabi-Yau threefold, 
in the sense that its $h^{1}$ vanishes and it has precisely a one-dimensional
space of holomorphic three-forms or covariantly constant spinors.
For the four cases with Hodge numbers $(3, 3)$,
the holonomy is a proper subgroup of $SU(3)$,
in fact a finite subgroup.
The remaining cases, where $h^{2,1} = 3, h^{1,1} > 3$, 
involve some blowups, and their holonomy is all of $SU(3)$.

Below we give a global description of the moduli space $M_{X}$  in each case in which $h^{2,1} = 3$. In doing so, we will also describe a connected component $M_{\overline{X}}$ (isomorphic to $M_X$) of the moduli space of complex structures on a particular resolution ${\overline{X}}$. What we will find is that $M_X$ is isomorphic to a global quotient of $\mathbb{H}^3$, where $\mathbb{H}$ is the complex upper half plane. We will prove that the Picard group of $M_X$ is a finitely generated abelian group, and calculate the rank of its free abelian part in each case. Despite the fact that $\text{Pic}(M_X)$ will often be infinite, we will prove that the Hodge bundle has finite order in each of these cases. In fact, we will calculate an explicit trivialization of a tensor power of the bundle, and compute the K\"ahler potential on $M_X$.

\newpage

\section{The moduli spaces}

\vskip .5cm

\subsection{Notation and Terminology} \label{notation}
\begin{eqnarray*}
V &:=& \mbox{a fixed, 2-dimensional vector space over the field of 2 elements}\\
S &:=& SL(V) \cong SL(2,{\mathbb Z}/2), \mbox{ acting linearly on } V.\\
B_i \subset S &:=& \mbox{a Borel subgroup, stabilizer of a non-zero vector in } V, \ i = 1,2,3.\\
E_i &:=&  \mbox{elliptic curves}, \  i=1,2,3. \\
E_i[2] &:=& \mbox{the subgroup of points of order 2 in } E_i\\
Y & := & E_1\times E_2\times E_3 \\
Y[2] &:=& \mbox{the subgroup of points of order 2 in } Y\\
l_i&:=&  \mbox{a level 2 structure on  } E_i,  \mbox{ i.e. an isomorphism } 
l_i:= V \stackrel{\sim}{\to} E_i[2] \\
l_{\mbox{ }} &:=& \Pi_{i=1}^3 l_i : V^{\oplus 3} \stackrel{\sim}{\to} Y[2], \mbox{ the induced level 2 structure on Y}\\
M(2) & := & \mbox{moduli stack of elliptic curves with level 2 structure} \\
M(4) & := & \mbox{moduli stack of elliptic curves with level 3 structure} \\
M & := & \mbox{moduli stack of elliptic curves} =M(2) /S \\
\Gamma&:=& SL(2,\Z)\\
A&:=& (\R/\Z)^2 \rtimes \Gamma \\
A^{(3)} & := & A^3/S_3 \\
G_{max} &:=& \mbox{Ker } ((V \times \Z/2)^3 \to ({\mathbb Z}/2)^3 \to {\mathbb Z}/2)\\
G_{\mbox{ }}&:=& \mbox{a subgroup of } G_{max} \mbox{ mapping onto Ker}  \left( ({\mathbb Z}/2)^3 \to {\mathbb Z}/2 \right)\\
G_{S} &:=& \mbox{Ker} \left(G \to ({\mathbb Z}/2)^3 \subset ({\mathbb Z}/2)^3 \right) 
\subset V^{\oplus 3}, \mbox{ the `shift' part of } G\\
G_{T} &:=& \mbox{a subgroup of $G$ mapping isomorphically to } G/G_S,
\mbox{the `twist' part of } G\\
H & := & N_{A^{(3)}}(G) \\
H' & := & H/G \\
H_{max} & := & (\Z/4)^6 \rtimes \Gamma^3 \rtimes S_3 \leq A^{(3)}, \mbox{where } (\Z/4)^6 \leq (\R/\Z)^6 \mbox{ is the set of points of order 4}.\\
\mathbb{H} & := & \mbox{Complex upper half plane} \\
\mathcal{U} &:=& \mbox{A family of elliptic curves } (\mathbb{H} \times \C)/\Z^2 \mbox{ over } \mathbb{H} \\
X &:=& Y/G \\
M_X^{+} & := & \mbox{moduli stack of complex structures on }X. \mbox{ (Its dimension is $h^{2,1}(X)$.)} \\
M_X & := & \mbox{The connected component of } M_X^{+} \mbox{ containing the actual quotients} \  X = Y/G. \\
\overline{X}&:= &  \mbox{a Calabi-Yau threefold, crepant resolution of }X. \\
M_{\overline{X}} & := & \mbox{moduli stack of the Calabi-Yau
threefolds $\overline{X}$ (Its dimension is again $h^{2,1}(X)$)}. \\
& & \mbox{It is a 
covering of $M_X$, 
the finite fiber parametrizing K\"ahler resolutions $\overline{X}$} \\
& & \mbox{of
a given $X$. This is the space we care about!}\\
M_{\overline{X},\text{central}} &:=& \mbox{A connected component of $M_{\overline{X}}$}.
\end{eqnarray*}

Here are some remarks regarding our terminology. Given a complex algebraic variety $X$, there is a moduli stack $M^{\text{alg}}_X$ parameterizing deformations of $X$ over pointed schemes (seperated and of finite type) over $\C$. More precisely, a deformation of $X$ over $(S, s_0)$ is a flat and proper morphism $f: \mathcal{X} \rightarrow S$, together with an isomorphism $\phi: f^{-1}(s_0) \cong X$. Note that we do not require the data of a polarization on our varieties (that is, a fixed ample invertible sheaf $\mathcal{L}$ on $\mathcal{X}$), so standard results about the existence of a quasi-projective coarse moduli scheme \cite{Vie95} would not be available. When $X$ is a smooth compact Calabi-Yau variety, the Tian-Todorov theorem (\cite{Ti87}, \cite{To89}) implies that $X$ has unobstructed deformations, and hence $M^{\text{alg}}_X$ is smooth. In the cases of this paper, $X$ will be Calabi-Yau with singularities in at worst codimension 2, in which case a result of Ran \cite{Ra93} implies that $M^{\text{alg}}_X$ is still smooth. 

The stack $M^{\text{alg}}_X$ determines an underlying analytic stack (see \cite{BN05}, section 3, for basic definitions regarding complex analytic stacks), which we will refer to as $M_X$ for the rest of this paper. For example, the moduli stack $M$ of elliptic curves over $\C$ has an underlying analytic stack, which is the quotient $[\mathbb{H}/SL(2, \Z)]$. Our goal is to provide a similar global description of $M_X$ and ultimately $M_{\overline{X}}$, for $\overline{X}$ a smooth crepant resolution of $X$. 

\subsection{More details: the spaces and the maps between them} \label{details}

\vskip .5cm

The moduli spaces we consider fit into a sequence of maps:
$$
\begin{tikzcd}
& M_{\overline{X}, \text{central}} \ar[d] \ar[hookrightarrow]{r} & M_{\overline{X}} \ar[ld]\\ 
\mathbb{H}^3 \ar[r, "f"] & M_{\overline{X}}
\end{tikzcd}
$$
We will describe each of these spaces and the maps between them.\\
$\mathbb{H}$ denotes the complex upper half plane. Over $\mathbb{H}$, there is a family $\mathcal{U}$ of elliptic curves
$$
\mathcal{U} = (\mathbb{H} \times \mathbb{C})/\Z^2
$$
together with a global section $s: \mathbb{H} \rightarrow \mathcal{U}$ given by $s(\tau) = (\tau,0)$. (Recall that an elliptic curve $(E,0)$, often abbreviated to just $E$, is a smooth genus 1 curve $E$
with a marked point $0 \in E$.) The fiber of $\mathcal{U}$ over a point $\tau \in \mathbb{H}$ is the complex torus $\C/\Lambda_{\tau}$ with the origin $s(\tau)$. Here, $\Lambda_{\tau}$ is the lattice in $\C$ generated by $1$ and $\tau$. If $M$ denotes the moduli stack of elliptic curves, the family $\mathcal{U}$ defines a map $\mathbb{H} \rightarrow M$. In fact, as is well known, $M$ is the stacky quotient of $\mathbb{H}$ by the group
$$
\Gamma = \text{SL}(2,\Z),
$$
which acts on $\mathbb{H}$ by fractional linear transformations. Note that over the space $\mathbb{H}^3$ we have the family $\mathcal{U}^3 \rightarrow \mathbb{H}^3$ of abelian varieties, which will be cruical for later developments. 

Throughout the paper, $G$ refers to an abelian group isomorphic to $({\mathbb Z}/2)^{r+2}$,
$0 \leq r \leq 6$, acting on the product $Y = E_1 \times E_2 \times E_3$ of three elliptic curves $\{E_i\}_{i=1,2,3}$ according to
one of the 36 entries in Table 1 of \cite{DW}.  $G$ is given as an extension
\begin{displaymath}
\xymatrix{
1 \ar[r] & G_S \ar[r] & G \ar[r] & G_T \ar[r] & 1,
}
\end{displaymath}
where the subgroup $G_S \approx ({\mathbb Z}/2)^r$ of ``shifts'' acts
on $Y$ by translation, so $G_S$ can be identified with a subgroup of
$Y[2]$, while each nontrivial element of the group $G_T \approx 
({\mathbb Z}/2)^2$ (of ``twists'') acts, modulo some translations, as 
inversion $y_i \mapsto - y_i$ on two of the three $E_i$.\\

Now, fix a group $G$ as in Table 1 of \cite{DW}. $G$ acts naturally on the family $\mathcal{U}^3 \rightarrow \mathbb{H}^3$ (acting trivially on the base). Therefore, we have an induced family of complex orbifolds
$$
\eta: \mathcal{U}^3/G \rightarrow \mathbb{H}^3.
$$
The fiber of $\eta$ over the triple $(\tau_i)_{i=1}^3$ is the toroidal orbifold
$$
\left(\prod_{i=1}^{3} \C/\Lambda_{\tau_i}\right)/G.
$$
Since $X$ appears as the fiber of $\eta$ over some triple $(\tau_1,\tau_2,\tau_3) \in \mathbb{H}^3$, $\eta$ determines an analytic map $f: \mathbb{H}^3 \rightarrow M_X$.  We will refer to the family $\mathcal{U}^3/G$ as $f^*\X$.

Finally, we let $\overline{X}$ denote any Calabi-Yau resolution of the
possibly singular $X$, and let $M_{\overline{X}}$ be the moduli space
of complex structures on $\overline{X}$. There is a forgetful map $M_{\overline{X}} \rightarrow M_X$, whose degree is equal to the number of K\"ahler resolutions of $X$. 

\vskip .5cm

\subsection{The plan} We will construct a discrete group $H'$ acting on $\mathbb{H}^3$. By lifting this action to $\mathcal{U}^3/G$, we will factor the map $f$ through a map $\bar{f}: [\mathbb{H}^3/H'] \rightarrow M_X$. We will then prove that $\bar{f}$ induces an isomorphism between $[\mathbb{H}^3/H']$ and the moduli stack $M_X$. Then, we will analyze the Picard groups of $M_X$ and $M_{\overline{X}}$, and study the Hodge bundle $\lambda$ in particular. 

\vskip .5cm

\section{Automorphisms of $f^*\X$}

\subsection{Group actions} In this section we will identify a group $H'$ acting on both $f^*\X$ and $\mathbb{H}^3$. It is this action that will induce the map $\bar{f}$ mentioned in section 2.3.

Firstly, the group $T^2 = (\R/\Z)^2$ acts on $\mathcal{U}$ by translations. To be precise, let $\epsilon = (\epsilon_0, \epsilon_1) \in T^2$ and $(\tau, z) \in \mathcal{U}$. Then 
$$
\epsilon \cdot (\tau, z) := (\tau, z + \epsilon_0 + \epsilon_1 \tau).
$$
As we have mentioned, there is an action of $\Gamma$ on $\mathbb{H}$. If $\gamma \in \Gamma$ is given by 
$$
\gamma = \begin{pmatrix}
a & b \\
c & d
\end{pmatrix},
$$
then
$$
\gamma \cdot \tau = \frac{a\tau + b}{c\tau + d}.
$$
This action of $\Gamma$ can be lifted to an action on $\mathcal{U}$ by defining
$$
\gamma \cdot (\tau, z) = \left(\frac{a\tau + b}{c \tau + d}, (c\tau + d)^{-1} z \right).
$$
This action of $\Gamma$ on $\mathcal{U}$ normalizes the previous action of $T^2$, in the sense that if $\gamma \in \Gamma$ and $\epsilon \in T^2$, then there exists an $\epsilon' \in T^2$ such that for every $(\tau, z) \in \mathcal{U}$ we have 
$$
\gamma \cdot \epsilon \cdot \gamma^{-1} \cdot (\tau, z) = \epsilon' \cdot (\tau, z).
$$ 
(Explicitly, $\epsilon' = a \epsilon_0  - b \epsilon_1 , - c \epsilon_0  +  d \epsilon_1  $.)
Therefore $\Gamma$ acts on $T^2$ by sending $\epsilon$ to $\epsilon'$, and we may form the extension
$$
\begin{tikzcd}
1 \arrow[r] & T^2 \arrow[r] & A \arrow[r] & \Gamma \arrow[r] &  1
\end{tikzcd}
$$
as a semi-direct product $A = T^2 \rtimes \Gamma$. The previous actions of $T^2$ and $\Gamma$ on $\mathcal{U}$ are now combined into a single action of $A$. \\

Now, $A^3$ acts diagonally on $\mathcal{U}^3$. There is also an $S_3$ action on $\mathcal{U}^3$ by permutation of the three factors, as well as an $S_3$ action on $A^3$, so we form the extension
$$
\begin{tikzcd}
1 \arrow[r] & A^3 \arrow[r] & A^{(3)} \arrow[r] & S_3 \arrow[r] &  1.
\end{tikzcd}
$$
So $A^{(3)}$ acts on $\mathcal{U}^3$, combining all of the previous actions. Since the group $G$ acts on $\mathcal{U}^3$ by translations of points of order 2 and elliptic inversions, it embeds naturally into $A^{(3)}$. Therefore we may form the normalizer
$$
H := N_{A^{(3)}}(G),
$$
and subsequently the quotient
$$
H' := H/G.
$$
Since $H$ normalizes $G$, its action on $\mathcal{U}^3$ descends to an action on $f^*\X$. We remark that since $G$ acts trivially on $\mathbb{H}^3$ and $f^*\X$, $H'$ acts on both of these spaces. We denote the respective actions of $H'$ on $\mathcal{U}^3$ and $f^*\X$ by $a$ and $\widetilde{a}$. 

\begin{prop}
Let $h = (\epsilon, \gamma, \sigma)$ denote an element of $A^{(3)}$, where $\epsilon \in T^6$, $\gamma \in \Gamma^3$, and $\sigma \in S_3$. Then if $h \in H$, we must have $\epsilon \in (\Z/4)^6$, the subgroup of points of $T^6$ of order dividing 4. 
\end{prop}
\begin{proof}
Let $g \in G$. For $hgh^{-1}$ to belong to $G$, its $T^6$ component must be a point of order 2. We can write $g \in A^{(3)}$ as $(\delta, \iota, 1)$ where $\iota$ is a pure twist ($-I$ on an even number of factors) and $\delta$ has order 2. Firstly, we observe that $(0,1,\sigma)g(0,1,\sigma^{-1})$ also has the form $(\delta', \iota', 1)$ for $\delta'$ a point of order 2 and $\iota'$ a pure twist. Therefore, we assume without loss of generality that $\sigma = 1$. Then we have
$$
hgh^{-1} = (\epsilon, \gamma, 1)(\delta, \iota,1)(-\gamma^{-1} \cdot \epsilon, \gamma,1).
$$
By assumption on $G$, there is a component of $\iota$ acting by $-I$. Referring to this component as the $i$th, we will have that the $i$th component of the above element is
$$
-(-\epsilon_i) + \delta_i + \epsilon_i = \delta_i + 2\epsilon_i.
$$
As we mentioned, this element must have order 2. Since $\delta_i$ has order 2, it follows that $\epsilon_i$ has order dividing 4.
\end{proof}

Let $H_{max}$ denote the $A^{(3)}$ subgroup given by $((\Z/4)^2 \rtimes \Gamma)^3 \rtimes S_3$. The previous proposition shows that $H \leq H_{max}$ in all cases, hence the name. 

Let $A(2)$ denote the subset of $A$ of elements of the form $(\epsilon, \gamma)$, where $\epsilon$ has order 2 and $\gamma \in \Gamma(2)$.

\begin{prop}
$A(2)^3$ is a normal subgroup of $H_{max}$, isomorphic to the direct product $(\Z/2)^6 \times \Gamma(2)^3$. Furthermore, $A(2)^3$ is always contained in $H$.
\end{prop}
\begin{proof}
Firstly, we must verify that $A(2)$ is a subgroup of $A$. We note that for any $(\tau, z) \in \mathcal{U}^3$ and $\gamma \in \Gamma(2)$, $\epsilon \in (\Z/2)^2$, an easy calculation shows that
$$
\gamma \cdot \epsilon \cdot \gamma^{-1} \cdot (\tau, z) = (\tau, z + \epsilon).
$$
Therefore, $A(2)$ is indeed closed under the product. Furthermore, the subgroups $(\Z/2)^2$ and $\Gamma(2)$ embed naturally into $A(2)$, and the above calculation shows that they commute. Since they clearly generate $A(2)$ and intersect trivially, we have shown that $A(2) \cong (\Z/2)^2 \times \Gamma(2)$. For normality, we begin by checking that $A(2)$ is normal in $(\Z/4)^2 \rtimes \Gamma$. let $\gamma' \in \Gamma$, $\epsilon' \in (\Z/4)^6$, and $\gamma \in \Gamma(2)$, $\epsilon \in (\Z/2)^6$. Firstly,
$$
(\gamma', 0) (\gamma, \epsilon)(\gamma'^{-1}, 0) = (\gamma'\gamma\gamma'^{-1}, \gamma' \cdot \epsilon).
$$
By the normality of $\Gamma(2)$ in $\Gamma$, $\gamma'\gamma\gamma'^{-1} \in \Gamma(2)$. Furthermore, $\gamma'$ must send order 2 points to order 2 points. Therefore, $\Gamma$ normalizes $A(2)$. Next, we consider
$$
(1, \epsilon')(\gamma, \epsilon)(1,-\epsilon') = (\gamma, - \gamma \cdot \epsilon' + \epsilon + \epsilon' ).
$$
Then $2(\gamma \cdot \epsilon' + \epsilon + \epsilon') = (\gamma \cdot (2\epsilon') + 2 \epsilon' )$. Since $\gamma \in \Gamma(2)$, it fixes points of order 2. Hence this expression is equal to $4\epsilon' = 0$. So we have established that $A(2)$ is normal in $(\Z/4)^2 \rtimes \Gamma$. Since $A(2)^3$ is clearly preserved by permutation, it follows that $A(2)^3 \unlhd H_{max}$, as claimed. Since every element of $\Gamma(2)$ fixes points of order 2 in $T^2$ by conjugation, we see that $G$ lies in the center of $A(2)^3$. In particular, $A(2)^3 \leq H$. 
\end{proof}

Based on the previous proposition, we can define a \textit{finite} group $L = H/A(2)^3$. $L$ is a subgroup of $L_{max} := ((\Z/2)^2 \rtimes S)^3 \rtimes S_3$, where $S = \text{SL}(2,\Z/2)$ was introduced earlier as the quotient $\Gamma/\Gamma(2)$.

\subsection{Computation of $L$} In this section we explain how to compute the finite group $L$, and record the results in all cases in which $h^{2,1}(X) = 3$. The technique is straightforward. Since $L$ is by definition the image of $H$ under the quotient map $H_{max} \rightarrow L_{max}$, we need to determine which elements $l \in L_{max}$ admit a lift  $\widetilde{l} \in H_{max}$ which normalizes $G$.  Note that if $\widetilde{l}$ and $\widetilde{l}'$ are two distinct lifts of $l$, then $\widetilde{l}$ normalizes $G$ if and only if $\widetilde{l}'$ normalizes $G$. This result follows since $\widetilde{l}^{-1}\widetilde{l}' \in A(2)^3$, which commutes with $G$. 

Therefore, one simply needs to iterate over the elements of $L_{max} = (\Z/2)^6 \rtimes S^3 \rtimes S_3$, choose a lift for each element, and determine whether or not this element normalizes $G$. Since $L_{max}$ and $G$ are finite, this algorithm can be very easily implemented on a computer. To subsequently determine $H$, one then finds the preimage of $L$ in $H_{max}$. 

We will not need a complete description of $L$ for the main results of the paper. The translations by points of order 2, given by $(\Z/2)^6$, form a normal subgroup of $L_{max}$. Let $p: L_{max} \rightarrow S^3 \rtimes S_3$ be the quotient map, and $L_0 := p(L) \leq S^3 \rtimes S_3$. Since $L_0$ is the group that we will need to know to compute the Picard group, we will record it in all cases. We will also record the group $H$ when it has a particularly simple form. We illustrate the relationships between the groups introduced so far in the following diagram.
$$
\begin{tikzcd}
N_{H_{max}}(G) = H \ar[hookrightarrow]{r} \ar[twoheadrightarrow]{d} & H_{max} \ar[twoheadrightarrow]{d} = (\Z/4)^6 \rtimes \Gamma^3 \rtimes S_3 \\
L \ar[hookrightarrow]{r} \ar[twoheadrightarrow]{d} & L_{max} = H_{max}/A(2)^3 \ar[twoheadrightarrow]{d} \cong (\Z/2)^6 \rtimes S^3 \rtimes S_3 \\
L_0  \ar[hookrightarrow]{r} & S^3 \rtimes S_3
\end{tikzcd}
$$
We fix some notation. $S$ acts faithfully on the vector space $V = (\mathbb{F}_2)^2$, so we will use this action to refer to elements of $S$. Namely, we label the nonzero elements of $V$ as $\{\frac{1}{2},\frac{\tau}{2}, \frac{1 + \tau}{2}\}$. There is an isomorphism $S \cong S_3$, as $S$ acts to permute these three nonzero elements. We let $B_i \leq S$ be the stabilizer of the $i$th element. The subgroup $\widetilde{B}_i \leq S^3$ refers to the set of triples with one identity element, and two elements of $B_i$. $S \leq S^3$ refers to the diagonal subgroup. We let $\Gamma_i$ and $\widetilde{\Gamma}_i$ refer to the preimages of $B_i$ and $\widetilde{B}_i$ under the quotient $\Gamma \rightarrow S$. The results, which we explained how to obtain earlier, are given in all cases with $h^{2,1} = 3$ except (4-1) by Table 1. \\

\begin{table}[ht]
    \centering
    \begin{tabular}{c|c|c}
        Case & $L_0 = p(L)$ & $H$ \\
        \hline
        (0-1) & $S^3 \rtimes S_3$ & $(\Z/2)^6 \rtimes \Gamma^3 \rtimes S_3$ \\
        (0-4) & $B_1^3 \rtimes S_3$ & $-$ \\
        \hline
        (1-1) & $B_2^3 \rtimes S_3$ & $(\Z/2)^6 \rtimes \Gamma_2^3 \rtimes S_3$ \\
        (1-5) & $\widetilde{B}_2 \rtimes S_3$ & $-$\\
        (1-11) & $(B_2^2 \times B_1) \rtimes \langle (1 \ 2)\rangle $ & $-$ \\
        \hline
        (2-1) & $S \rtimes S_3$ & $(\Z/2)^6 \rtimes \Gamma \rtimes S_3$\\
        (2-9) & $B_1^3 \rtimes S_3$ & $(\Z/2 \times \Z/4)^3 \rtimes \Gamma_1^3 \rtimes S_3$\\
        (2-12) & $(1 \times B_1^2) \rtimes \langle (2 \ 3)\rangle$ & $-$\\
        \hline
        (3-5) & $\widetilde{B}_1 \rtimes S_3$ & $(\Z/2 \times \Z/4)^3 \rtimes \widetilde{\Gamma} \rtimes S_3$
    \end{tabular}
    \caption{The group $L_0$ in cases with $h^{2,1} = 3$}
    \label{tab:my_label}
\end{table}
For case (4-1), $L_0$ is not a semi-direct product. We will note that $L_0$ fits into an extension
$$
\begin{tikzcd}
1 \ar[r] & N \ar[r] & L_0 \ar[r] & S_3 \ar[r] & 1,
\end{tikzcd}
$$
where $N \leq S^3$ has the property that each projection $\pi_i: S^3 \rightarrow S$ induces an isomorphism $N \cong S$. However, $N$ is not the diagonally embedded copy of $S$ in $S^3$.

\section{Global geometry of $M_X$}

We have described a map $f: \mathbb{H}^3 \rightarrow M_X$. In order to obtain a map $\bar{f}: [\mathbb{H}^3/H'] \rightarrow M_X$, we must lift the action of $H'$ on $\mathbb{H}^3$ to an action on the family $f^*\X$. Of course, we have also done this by defining the action $\widetilde{a}$. Our objective now is to prove that $\bar{f}$ is an isomorphism. The strategy will be to show that $\bar{f}$ is proper and \'etale, and then prove that it has degree 1. We begin with a few lemmas. 

\begin{lem}   \label{prop:pairwise:lifts}
Assume that three elliptic curves $E_i$ are pairwise non-isogenous, $Y$ is their product,  
$\pi: Y \to X:= Y/G$ the quotient by a $G$-action on $Y$, and
$\pi': Y' \to X'$ another quotient in the same family (i.e. with same $G$).
Then any isomorphism $X \to X'$ lifts
to an isomorphism $Y \to Y'$.
\end{lem}
\begin{proof}
Note that if $\{i,j,k\}$ is a permutation of $\{1,2,3\}$, 
the action of $G$ on $Y$ induces an action on
$E_i \times E_j$, and $X$ maps to the quotient 
$(E_i \times E_j)/G$. The generic fiber is isomorphic to $E_k$. Indeed,
any element of $G$ that acts trivially on $E_i \times E_j$ must be a shift, 
in $G_S$. But the reduction explained
in \cite{DW} allows us to assume that the group $G$ of automorphisms of $X$ 
is essential, or non-redundant,
meaning that it does not contain a translation by a nonzero $x \in E_k$.

So our $X$ has three elliptic fibrations, 
which are distinct because the fibers $E_1, E_2, E_3$ are 
assumed non-isogenous. Further, these are the  {\it only} 
elliptic fibrations on $X$: in fact, their lifts
to $Y$ are the only elliptic fibrations there. 
To see this, consider a genus 1 fibration $\pi: Y \to B$, 
and let $E$ be a generic fiber. 
Note that the fixed locus of elements of $G$ is at most one-dimensional, 
so $\pi^{-1}$ of the generic $E$ does not meet the fixed locus. 
The inverse image $\pi^{-1}(E)$ of $E$ in $E_1\times E_2\times E_3$ 
is then an unramified cover of $E$, 
so each of its connected components $E'$ is itself an elliptic curve in 
$E_1\times E_2\times E_3$. 
The assumption about non-isogeny of the $E_i$ implies that $E'$ 
can map onto at most one of the $E_i$. 
It follows that $E'$ is isomorphic to one of the $E_i$ 
and is embedded in the expected way: parallel to one of the three coordinates. 
So the same holds for $E$.

The isomorphism $X \to X'$ must therefore take an elliptic fibration of $X$ 
with fiber $E_i$ to 
an elliptic fibration of $X'$ with the same fiber $E_i$. 
It therefore lifts to an isomorphism 
$Y \to Y'$ as claimed.
\end{proof} 

\begin{rem}
Note that the isomorphism $Y \rightarrow Y'$ obtained at the end of the theorem above is only a complex isomorphism, \textit{not} an isomorphism of abelian varieties. That is, it may fail to map the origin of $Y$ to that of $Y'$. It is precisely for that reason that the group $H_{max}$, and hence $H$, contains a translation component $(\Z/4)^6$. 
\end{rem}

\begin{lem}
Assume that three elliptic curves $E_i$ are pairwise non-isogenous, $Y$ is their product,  
$\pi: Y \to X:= Y/G$ the quotient by a $G$-action on $Y$, and
$\pi': Y' \to X'$ another quotient in the same family (i.e. with same $G$).
Let $\phi: X \to X'$ be an isomorphism. Furthermore, assume that we have fixed isomorphisms $\psi_1: \eta^{-1}(\tau) \cong X$ and $\psi_2: \eta^{-1}(\tau') \cong X'$, for some $\tau, \tau' \in \mathbb{H}^3$. Then there exists a unique $h' \in H'$, such that $$
\widetilde{a}(h') \vert_{\eta^{-1}(\tau)} \circ \psi_1 = \psi_2 \circ \phi.
$$
That is, the restriction of the action of $h'$ on $f^*\X$ to $\eta^{-1}(\tau)$ induces the isomorphism $\phi$.
\end{lem}
\begin{proof}
By the above lemma, we have an element $\sigma \in S_3$ together with complex isomorphisms $\phi_i: E_i \xrightarrow{\sim} E_{\sigma(i)}$, such that $\widetilde{\phi} = \prod_{i=1}^3 \phi_i$ induces $\phi$. Any isomorphism of two elliptic curves has a lift to an automorphism of the universal cover, $\C$. Since any automorphism of $\C$ is an affine transformation $z \mapsto az + b$, we can find a unique point $\epsilon_i \in T^2$ such that $\epsilon_i^{-1} \circ \phi_i$ is origin preserving. In that case, $\epsilon_i^{-1} \circ \phi_i$ is an isomorphism of elliptic curves, and is therefore induced uniquely by the action of some $\gamma \in \Gamma$ on the universal family $\mathcal{U}$ (since $M \cong [\mathbb{H}/\Gamma]$). Therefore, the action of $h = (\epsilon, \gamma, \sigma) \in A^{(3)}$ on $f^*\X$ induces the isomorphism $\phi$. Since the action of $h$ descends to $f^*\X = \mathcal{U}^3/G$, we must have $h \in N_{A^{(3)}}(G) = H$.

Now for uniqueness. An element $h$ will act trivially on a fiber of $f^*\X$ if and only if it belongs to $G$. To see this latter claim, assume that $h = (\epsilon, \gamma, \sigma)$ acts trivially on the fiber $X \cong \eta^{-1}(\tau) \subset f^*\X$. The action of $h$ on $X$ comes from its action on $Y$, a fiber of $\mathcal{U}^3$. Since $h$ acts trivially on $X$, we must have for any $y \in Y$ that there exists a $g_y \in G$ such that $g_y \cdot y = h \cdot y$. Since $h \cdot y$ depends smoothly on $y$, so must $g_y \cdot y$. Since $G$ is finite, we then have $g_y = g$, a constant. But then $hg^{-1}$ acts trivially on $Y$. This clearly implies that $hg^{-1} = 1$, and $h \in G$. Taking $h'$ to be the image of $h$ in $H'$ then proves the claim. 
\end{proof}
Let $O \subset \mathbb{H}^3$ be the subset of triples $(\tau_i)_{i=1}^{3}$ with the propery that the three elliptic curves $E_{\tau_i} = \C/\Lambda_{\tau_i}$ are non-isogeneous. $O$ is a dense subset of $\mathbb{H}^3$. Since $O$ is stable under the action of $H'$, its image $\mathcal{O}$ in the quotient $[\mathbb{H}^3/H']$ is then a dense open substack.

\begin{thm}
The map $\bar{f}: \mathcal{O} \rightarrow M_X$ is a monomorphism of stacks.
\end{thm}
\begin{proof}
We must show that $\bar{f}$ is a fully faithful functor. That is, assume that we are given an isomorphism 
$$
\begin{tikzcd}
\phi^*(f^*\X) \ar[r,"\sim"] \ar[d] & \psi^*(f^*\X) \ar[d] \\
S \ar[r,"\sim"] & S'
\end{tikzcd}
$$
of pullbacks of the family $f^*\X$ to complex spaces $S$ and $S'$ via maps $\phi: S \rightarrow \mathbb{H}^3$ and $\psi: S' \rightarrow \mathbb{H}^3$. Then we must show that there exists a unique $h' \in H'$ such that the following diagram commutes.
$$
\begin{tikzcd}
\phi^*(f^*\X) \ar[r, "\sim"] \ar[d] &\psi^*(f^*\X) \ar[d] \\
f^*\X \ar[r, "\widetilde{a}(h')"] & f^*\X
\end{tikzcd}
$$
Fix $s \in S$ (not necessarily the basepoint). Let $X_s$ be the fiber of $\phi^{*}(f^*\X)$ over $s$. Let $s'$ be the image of $s$ in $S'$, and $X'_s$ the fiber of $\psi^*(f^*\X)$ over $s'$. Then $\phi$ (resp. $\psi$) maps $X_s$ (resp. $X'_s$) isomorphically onto some fiber $X$ (resp. $X'$) of $f^*\X$. Therefore, we obtain an induced isomorphism $X \rightarrow X'$. But by Lemma 4.2, this means that we can find an element $h'_s \in H'$ whose action on $f^*\X$ induces the isomorphism $X \rightarrow X'$. We can construct such an $h'_s$ for every $s \in S$. But since $H'$ is discrete and acts continuously, we must have $h'_s = h'$, a constant. Therefore, the claim follows. 

\end{proof}

\begin{rem}

For the purpose of proving that $\bar{f}$ is \'etale and proper, it is permitted to replace $[\mathbb{H}^3/H']$ by a finite \'etale cover. We have a subgroup $\Gamma(4)^3 \leq H$. This subgroup intersects $G$ trivially, since $-I$ is not congruent to $I$ mod $4$. Therefore, the quotient map $H \rightarrow H'$ induces an isomorphism between $\Gamma(4)^3$ and some subgroup of $H'$, namely $\Gamma(4)^3 G /G$. In this way, we can regard $\Gamma(4)^3$ as a subgroup of $H'$, and hence $M(4)^3$ as a finite \'etale cover of $[\mathbb{H}^3/H']$. From now on, we will work with the induced map $M(4)^3 \rightarrow M_X$

\end{rem}

\begin{thm}
The induced map $M(4)^3 \to M_X$ is an immersion, 
and in the 10 cases when $h^{2,1}=3$
it is an open embedding.
\end{thm}
\begin{proof}
Consider the sum of the the three lines 
\[
H^{1,0}(E_i) \otimes H^{1,0}(E_j) \otimes H^{0,1}(E_k)
\]
as $\{i,j,k\}$ runs over cyclic permutations of $\{1,2,3\}$. 
The orbifold map $E_1 \times E_2 \times E_3 \to X$ 
embeds the sum of  these three lines 
as the 'bulk sector,' a direct summand of $H^{2,1}(X)$. 
The assumption $h^{2,1}=3$ implies that in those cases,
the sum is actually isomorphic to $H^{2,1}(X)$. 
By contracting with top forms we switch to tangent spaces, finding that  
the sum of the three $H^1(E_i, T_{E_i})$ equals 
(all of or a direct summand of) $H^1(X,T_X)$.
This shows that $\bar{f}$ is an immersion in general, 
and in the 10 cases when $h^{2,1}=3$
it is a local isomorphism. On the other hand, 
Theorem 4.3 along with the remark above shows that 
$M(4)^3 \rightarrow M_X$ is injective on the complement of a countable union of divisors in 
$M(4)^3$, 
so when $h^{2,1}=3$ it must be an open embedding. 
(In the remaining cases, the map is a birational morphism to its image 
in $M_X$.)
\end{proof}

\begin{thm}
The induced map $M(4)^3 \to M_X$ is proper, 
so in general is a desingularization of its image 
which is a closed subspace of $M_X$.
\end{thm}
\begin{proof}
Over $M(4)^3$  there are three universal bundles 
whose fibers, respectively, are the three elliptic $E_i$.
By the valuative criterion for properness, we need to show that 
if we have an algebraic family $X_t$ of $X$'s parametrized by $t$ in 
a regular algebraic curve $\Delta$
(or, intuitively, in the unit disc)
and a lift $Y_t \to X_t$ over the generic point of $\Delta$
(respectively, a family of lifts parametrized by $t$ in the punctured disc)
then we can complete the curve of lifts
(respectively, fill in with a $Y_0 \to X_0$).
The lift gives us three maps $E_i$ from the punctured disc to $M(4)$, hence to $M$.
We see by the same valuative criterion that each of these maps 
extends to the compactification:
$E_i: \Delta \to {\overline{M}}$.
We claim that the central fibers $E_{i,0}$ must be non-singular. 
This follows immediately from the canonical identification \eqref{product}
of $H^3(X_t,\Q)$ with $\otimes_{1=1,2,3}H^1(E_{t,i},\Q)$: 
The monodromy on $H^3(X_t,\Q)$, as $t$ goes around the punctured disc,
is trivial since the family extends over $\Delta$.
So it must be trivial also on each $H^1(E_{t,i},\Q)$, 
so the elliptic curves cannot degenerate.
The family of $Y_t, \ t \in \Delta$ is therefore topologically trivial,
so the $G$ action extends to $Y_0$ with quotient $X_0$, proving the theorem.
\end{proof}
\begin{cor}
The map $\bar{f}: [\mathbb{H}^3/H'] \rightarrow M_X$ is an isomorphism in the 10 cases in which $h^{2,1} = 3$.
\end{cor}
\begin{proof}
We have shown that $\bar{f}$ is \'etale, hence unramified. Since it has degree 1 on the substack $\mathcal{O}$, the same must be true everywhere, and hence $\bar{f}$ is an open embedding. But we have shown that it is also proper. By the connectedness of $M_X$, $\bar{f}$ is an isomorphism.
\end{proof}
\section{Crepant resolutions and the global geometry of $M_{\overline{X}}$}

In this section we briefly review the crepant resolutions
${\overline{X}} \to X$, such that ${\overline{X}}$ is Calabi-Yau.  
In the four cases with $h^{1,1}=h^{2,1}=3$ (namely (0-4),
(1-5), (1-11), (2-12)), $G$ acts freely
on $Y$, so no resolution is needed.  However, the other six examples 
of $h^{2,1}=3$ all
have singularities in their quotient spaces.

\subsection{Local structure}

The fixed-point loci of our orbifolds $X$
are curves, possibly with trident singularities.
There are two local pictures.
At a singular point of a curve of fixed points, 
$X$ is locally of the form ${\mathbb C}^3/( {\mathbb Z}/2 )^2$, with generators
acting as
\begin{eqnarray*}
g_1: \: (x,y,z) & \mapsto & (x,-y,-z), \\
g_2: \: (x,y,z) & \mapsto & (-x,y,-z).
\end{eqnarray*}
The curve of fixed points looks locally like the trident $xyz=0$,
with the three branches $C_1, \ C_2, \ C_3$ 
consisting of fixed points of $g_1, \ g_2$, and $g_3:=g_1 g_2$.
(We can always arrange that the local curve $C_i$ is parallel 
to the global elliptic factor $E_i$ of $Y$.)
At a smooth point of the curve of fixed points,  
$X$ is locally of the form ${\mathbb C}^3/( {\mathbb Z}/2 )$, 
with generator one of the $g_i$. 
So $X$ looks there like the product of 
$C_i$ (=one of the three axes)
with an $A_1$ surface singularity in the two other directions.

We will focus on the first case, ${\mathbb C}^3/( {\mathbb Z}/2 )^2$.
Its coordinate ring is 
\begin{equation}   \label{app:orig-singularity}
{\mathbb C}[x^{2}, y^2, z^{2}, xyz] \: = \: {\mathbb C}[a,b,c,d] / 
(a b c = d^2).
\end{equation}
It has four crepant resolutions.  One way to see this is to take advantage of
the fact that this singularity is a toric variety, and we can represent
the flops graphically as cross-sections of a toric fan as follows:
\begin{center}
\includegraphics[width=0.65\textwidth]{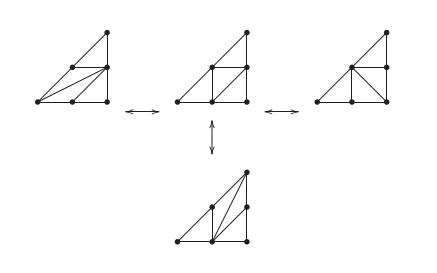}
\end{center}
Local coordinates on the central resolution are given by:
\begin{center}
\begin{tabular}{cc}
\includegraphics[width=0.18\textwidth]{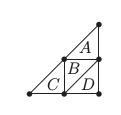}
&
$\begin{array}{l}
A \: = \: {\mathbb C}[b, c, d_{11}], \\
B \: = \: {\mathbb C}[c_{11}, d_1, t], \\
C \: = \: {\mathbb C}[a, b, w], \\
D \: = \: {\mathbb C}[a, b_1, c],
\end{array}$
\end{tabular}
\end{center}
where $tw = 1$, $c_{11} d_{11} = 1$, $b_1 d_1= 1$,
while local coordinates on one of the outlying resolutions are given by:
\begin{center}
\begin{tabular}{cc}
\includegraphics[width=0.18\textwidth]{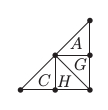}
&
$\begin{array}{l}
A \: = \: {\mathbb C}[b, c, d_{11}], \\
G \: = \: {\mathbb C}[b, c_{11}, t], \\
H \: = \: {\mathbb C}[a, d_1, w], \\
C \: = \: {\mathbb C}[a, b_1, c].
\end{array}$
\end{tabular}
\end{center}

The three outlying crepant resolutions 
as well as the central resolution
can be described as successive blowups.
Let $i,j,k$ be a permutation of $1,2,3$.
First blow $X$ up along a smooth surface containing $C_i$ and $C_j$, 
two of the components of the curve of fixed points in $X$.
(E.g. for $i,j,k = 1,3,2$, the surface is $b=d=0$.)
Second, blow up 
along a smooth surface containing
the remaining curve $C_k$. 
(In the above case, we could take this second surface to be $c=d/b=0$.)
This turns out to leave an isolated singularity, a conifold.
So, third and finally, take a crepant resolution
of the conifold.  There are two such resolutions; 
one will be 
the $k$-th outer resolution,
and the other, independently of the permutation $i,j,k$, 
will be the central resolution.

\subsection{Global structure}

Globally there are further subtleties to get a K\"ahler
resolution.  Consider for example the case (0-1).  This has curves of
$A_1$ singularities (along sixteen copies of each of the three elliptic curves $E_i$).
These intersect at 64 fixed points, each of the form
of the $({\mathbb Z}/2)^2$ quotient above.
 There is thus a total of $4^{64}$
possible resolutions.  However, not all of these are K\"ahler.  
We have not checked the total number of K\"ahler resolutions; it is probably large 
\footnote{P.~Aspinwall conjectures that it is on the order of $2^{48}$.}.

In principle, whatever the number $n$ of K\"ahler crepant resolutions,
the forgetful map  $M_{\overline{X}} \to M_X$ is a covering of this degree $n$.

Note that preimage of $M_X$ in $M_{\overline{X}}$ is itself reducible: 
it has a central component $M_{\overline{X}, \text{central}}$, 
specified by performing the central local resolution at each point of the zero-dimensional stratum. 
It has one, two or three additional components of small degrees over $M_X$,
specified by performing the type-$i$ local resolution 
at each point $p$ of the zero-dimensional stratum. 
Recall that each local curve $C_i$ at each point $p$  is parallel 
to the global elliptic factor $E_i$ of $Y$.
So on a given $X=Y/G$ it makes sense to perform the same type-$i$ local resolution 
at all points $p$.
On the other hand, as we vary $X$ in its moduli $M_X$,
monodromy may permute the $E_i$:
depending on the image $H_X/H'_X$ of $H_X$ in $S_3$,
all three of these choices may therefore be on the same component 
(of degree three over $M_X$), 
or on separate components (of degree one),
or two can come together with the third staying separate.

In any event, the map  $M_{\overline{X}, \text{central}} \to M_X$ of connected components is an isomorphism 
in all ten cases with $h^{2,1}=3$.

\section{Picard groups and Hodge bundles}

In this section we will study the Picard groups and Hodge line bundles
of $M_X$ when $h^{2,1}=3$. By Theorem 4.5, we have that $M_X \cong [\mathbb{H}^3/H']$. Our strategy is to prove that the structure sheaf $\mathcal{O}_{M_X}$ is acyclic, and apply the exponential sequence on $M_X$. We will then prove that $\text{Pic}(M_X)$ is isomorphic to the group cohomology $H^2(H', \Z)$, which we will analyze in detail. 

\subsection{Picard Groups} We begin with a few lemmas on analytic quotient stacks. The first is straightforward and serves mainly to remind the reader of some terminology.

\begin{lem}
Let $X$ be a smooth complex manifold, and $T$ a group acting virtually freely and holomorphically on $X$ (\textit{i.e.}, $T$ has a finite index subgroup $T_0$ acting freely on $X$). Let $\mathcal{X}$ denote the quotient stack $[X/T]$. Then the categorical quotient $X_{mod} = X/T$ exists as a complex analytic space, and the natural map $f: \mathcal{X} \rightarrow X_{mod}$ is a coarse moduli space for $\mathcal{X}$. Furthermore, $f_* \mathcal{O}_{\mathcal{X}} \cong \mathcal{O}_{X_{mod}}$.
\end{lem}
\begin{proof}
We have in particular that $T$ acts properly discontinuously on $X$. Then \cite{Car57}, $X_{mod} = X/T$ is a normal analytic space. The map $f: \mathcal{X} \rightarrow X_{mod}$ is induced by the natural map $X \rightarrow X_{mod}$, since the latter is $T$ equivariant. To see that the map is a coarse moduli space, note that any map $\phi$ from $\mathcal{X}$ to a space $Y$ factors through a map $X/T \rightarrow Y$, since any such map $\phi$ is by definition a $T$ equivariant map $X \rightarrow Y$. Furthermore, it is immediate that there is a bijection between the closed points of $\mathcal{X}$ and $X/T$. 

For the last claim, note that a section of $f_* \mathcal{O}_{\mathcal{X}}$ over an open subset $U \subset X_{mod}$ is by definition a holomorphic function $\pi^{-1}(U) \rightarrow \C$ (for $\pi$ the quotient map $X \rightarrow X_{mod}$) which is $T$ invariant. However, by the universal property of the quotient, this is the same thing as a section of $\mathcal{O}_{X_{mod}}$ over $U$.
\end{proof}
We recall the following definition.
\begin{defn}
\cite{GR04} A complex analytic space $(S, \mathcal{O}_S)$ is called \textit{Stein} if it satisfies Cartan's Theorem B, namely every coherent analytic sheaf of $\mathcal{O}_S$-modules on $S$ is acyclic.
\end{defn}

\begin{lem}
Keeping the notation from the previous lemma, let $\mathcal{O}_{\mathcal{X}}$ denote the structure sheaf of $\mathcal{X}$. If $X_{mod}$ is a Stein space, then $\mathcal{O}_{\mathcal{X}}$ is acyclic (\textit{i.e.} $H^i(\mathcal{X}, \mathcal{O}_{\mathcal{X}}) = 0$ for $i > 0$).
\end{lem}
\begin{proof}
The group $H^i(\mathcal{X},\mathcal{O}_{\mathcal{X}})$ is by definition the $i$th equivariant sheaf cohomology of $\mathcal{O}_X$ over $X$ with respect to the action of $T$. In other words, it is the $i$th right derived functor of the composition $(-)^T \circ \Gamma$, where $\Gamma$ is the global sections functor and $(-)^T$ is the $T$-invariants functor. Therefore, we have a spectral sequence
$$
E^{p,q}_2 = H^p(T, H^q(X, \mathcal{O}_X)) \implies H^{p+q}(\mathcal{X}, \mathcal{O}_X).
$$
Firstly, assume that $T$ is finite. Since $H^q(X, \mathcal{O}_X)$ is a $\C$ vector space, we have that $E_2^{p,q} = 0$ for $p > 0$. In particular, there is an isomorphism
$$
H^i(\mathcal{X}, \mathcal{O}_{\mathcal{X}}) \cong H^0(T, H^2(X, \mathcal{O}_X)) \cong H^i(X, \mathcal{O}_X)^T .
$$
Now, since $T$ is finite, the quotient map $\pi: X \rightarrow X/T$ is a finite holomorphic mapping. Therefore by \cite{GR04} (Section I.1, Theorem 5), $\pi_*$ induces an isomorphism on coherent sheaf cohomology and 
$$
H^i(X, \mathcal{O}_X) \cong H^i(X_{mod}, \pi_*\mathcal{O}_{X}).
$$
Furthermore, since $\pi$ is proper, we also have that $\pi_*\mathcal{O}_X$ is a coherent sheaf of $\mathcal{O}_{X_{mod}}$ modules (\cite{GR04}, Section I.3, Theorem 3). Since $X_{mod}$ is Stein, Cartan's Theorem B ensures that $\pi_*\mathcal{O}_{X}$ is then acyclic. Putting everything together, we obtain $H^i(\mathcal{X}, \mathcal{O}_{\mathcal{X}}) = 0$ for $i > 0$.

Now, relax the assumption that $T$ is finite. By replacing $X$ by $X/T_0$ and $T$ by $T/T_0$, we apply the result above to obtain the claim.
\end{proof}

We return now to $M_X$. As a consequence of the first lemma, the analytic space $\mathbb{H}^3/H'$ is a coarse moduli space for $M_X$. We denote it by $C_X$.

\begin{lem}
The analytic space $C_X$ is a Stein space.
\end{lem}
\begin{proof}
Firstly, we remark that $C_X$ is isomorphic to $\mathbb{H}^3/H$, since $G$ acts trivially on $\mathbb{H}^3$. Since $H$ is a finite index subgroup of $H_{max} = (\Z/4)^6 \rtimes \Gamma^3 \rtimes S_3$, we have a finite holomorphic mapping $C_X \rightarrow C_{max}$, where $C_{max}$ is the quotient $\mathbb{H}^3/H_{max}$. Let us determine $C_{max}$. We already know that $\mathbb{H}/\Gamma \cong \C$, via the $j$-invariant. Since $(\Z/4)^6$ acts trivially on $\mathbb{H}^3$, we are reduced to computing $\mathbb{C}^3/S_3$, \textit{i.e.} the third symmetric power $\C^{(3)}$. But it is well known that $\C^{(3)} \cong \C^3$ (for example, by regarding the $\C^3$ as the vector space of monic degree 3 polynomials, the map sending $(z_1,z_2,z_3) \in \mathbb{C}^3$ to $(z - z_1)(z - z_2)(z - z_3) \in \mathbb{C}^3$ is a quotient map). We remark that $C_{max}$ is Stein. But since $C_X \rightarrow C_{max}$ is a finite holomorphic mapping, we conclude \cite{GR04} that $C_X$ is Stein.
\end{proof}

\begin{prop}
The sheaf $\mathcal{O}_{M_X}$ is acyclic.
\end{prop}
\begin{proof}
By Lemma 6.3, $C_X$ is Stein. By Lemma 6.2, the proposition follows.
\end{proof}
Now we can prove the main theorem of this section.
\begin{thm}
In each case with $h^{2,1} = 3$, we have $\text{Pic}(M_X) \cong H^2(H', \Z)$, where the latter is the second group cohomology of $H'$ with trivial action on $\Z$.
\end{thm}
\begin{proof}
On $M_X$, we have the exponential sequence of sheaves
$$
\begin{tikzcd}
1 \arrow[r] & \Z \arrow[r] & \mathcal{O}_{M_X} \arrow[r] & \mathcal{O}_{M_X}^* \arrow[r] & 1.
\end{tikzcd}
$$
This sequence induces a long exact cohomology sequence, which in particular contains the following section.
$$
\begin{tikzcd}[column sep=small]
\cdots \arrow[r] &  H^1(M_X, \mathcal{O}_{M_X}) \arrow[r] & H^1(M_X, \mathcal{O}_{M_X}^*) \arrow[r, "c_1"] & H^2(M_X, \Z) \arrow[r] & H^2(M_X, \mathcal{O}_{M_X}) \arrow[r] & \cdots 
\end{tikzcd}
$$
By Proposition 6.4, $H^i(M_X, \mathcal{O}_{M_X}) = 0$ for $i > 0$. Hence by the exactness of the sequence, the map $c_1$ is an isomorphism. So
$$
\text{Pic}(M_X) \cong H^2(M_X, \Z).
$$
Now, $M_X$ is the quotient stack $[\mathbb{H}^3/H']$. Therefore, there is an isomorphism
$$
H^2(M_X, \Z) \cong H^2(\mathbb{H}^3, H', \Z),
$$
where the latter group is the equivariant sheaf cohomology of $\mathbb{H}^3$ with respect to the sheaf $\Z$ and group $H'$. Since $\mathbb{H}^3$ is contractible, we have $H^i(\mathbb{H}^3,\Z) = 0$ for $i > 0$, and hence there is an isomorphism (\cite{Gro57}, Proposition 5.2.5) 
$$
H^2(\mathbb{H}^3, H', \Z) \cong H^2(H', \Z),
$$
\end{proof}
Now we begin to analyze the group $H^2(H', \Z)$. 
\subsection{Group cohomology}  Let us recall some basic elements of group cohomology. A reference for group cohomology is \cite{Br82}. Throughout this subsection, $G$ is an arbitrary group.

Recall that for any group $G$, the group cohomology functor $H^i(G,-): \text{Mod}_G \rightarrow \text{Ab}$ (for $\text{Mod}_G$ the category of $G$ modules, and $\text{Ab}$ the category of abelian groups) is defined as the $i$th right derived functor of $(-)^G: \text{Mod}_G \rightarrow \text{Mod}_G$, which sends a $G$ module $M$ to its $G$ invariant submodule $M^G$ (we may regard $(-)^G$ as a functor into the category of abelian groups, since $G$ acts trivially on $M^G$). Say that $G$ lies in an extension
$$
\begin{tikzcd}
1 \ar[r] & N \ar[r] & G \ar[r] & Q \ar[r] & 1.
\end{tikzcd}
$$
Then for any $G$-module $M$, $M^N$ acquires the structure of a $Q$ module. Indeed, $(-)^G = (-)^Q \circ (-)^N$. Therefore, the Grothendieck spectral sequence in this case becomes
$$
E_2^{p,q} = H^{p}(Q, H^q(N, M)) \implies H^{p + q}(G,M).
$$
This spectral sequence is known as the \textit{Lyndon-Hochschild-Serre} sequence. The associated 5-term exact sequence is 
$$
\begin{tikzcd}[column sep=small]
1 \ar[r] & H^1(Q,M^N) \ar[r] & H^1(G,M) \ar[r] & H^1(N,M)^{Q} \ar[r] & H^2(Q, M^N) \ar[r] & H^2(G,M),
\end{tikzcd}
$$
which is known as the \textit{inflation-restriction} sequence. The map $H^1(K,M^N) \rightarrow H^1(G,M)$ is known as the \textit{inflation map} and $H^1(G,M) \rightarrow H^1(N,M)^Q$ is known as the \textit{restriction map}. 

For any abelian group $M$, we can form $H^i(G,M)$ by equipping $M$ with the trivial $G$ action. If no action of $G$ is specified, this is what we will mean.

We also recall two useful facts. If $G$ acts trivially on $M$, then $H^1(G,M) \cong \text{Hom}(G,M)$. Also, if $M$ is an $G$ module which happens to be a $K$ vector space, where $K$ is a field such that $\text{Char}(K)$ does not divide the order of $G$, then $H^i(G,M) = 0$ for $i > 0$.

\subsection{Computing $H^2(H',\Z)$} In this section, we will carry out the main group theoretic calculation.

\begin{lem}
$H^2(H', \Z)$ is a finitely generated abelian group.
\end{lem}
\begin{proof}
We begin with a general statement. Let $N$ be a normal subgroup of a group $K$, with quotient $Q$. Assume that $Q$ is finite, and that the integral cohomology groups of $N$ are finitely generated. Then we show that $H^i(K, \Z)$ is finitely generated. We have the cohomological Hochschild-Serre spectral sequence:
$$
E_2^{p,q} = H^p(Q, H^q(N, \Z)) \implies H^{p + q}(K, \Z)
$$
Since $H^{q}(N,\Z)$ is finitely generated by hypothesis, and $Q$ is finite, the groups $E^{p,q}_2$ are finite for any $p,q \in \Z$. Therefore, $H^{p+q}(K,\Z)$ admits a finite step filtration with finitely generated quotients. It follows that $H^{p+q}(K,\Z)$ is finitely generated (since finitely generated abelian groups form a Serre class). 

Applying the lemma to the sequence
$$
\begin{tikzcd}
1 \arrow[r] & A(2)^3 \arrow[r] & H \arrow[r] & L \arrow[r] & 1, 
\end{tikzcd}
$$
we conclude that $H$ has finitely generated integral cohomology. Then we consider the sequence
$$
\begin{tikzcd}
1 \arrow[r] & G \arrow[r] & H \arrow[r] & H' \arrow[r] & 1. 
\end{tikzcd}
$$
The inflation map $H^i(H',\Z) \rightarrow H^i(H,\Z)$ has cokernel isomorphic to $H^i(G,\Z)^{H'}$, which is a finite group. The kernel is isomorphic to some subgroup of $H^{i-1}(G,\Z)$, which is also finite. Since $H^i(H,\Z)$ is finitely generated, we conclude that $H^i(H',\Z)$ is finitely generated as well.
\end{proof}
The previous lemma allows us to decompose $H^2(H', \Z)$ into a free part, $\Z^{k}$, and a torsion part. By the universal coefficient theorem, we can calculate the rank $k$ as the dimension of the vector space $H^2(H', \Q)$ over $\Q$. For the rest of this section, let $K$ be a field of characteristic not equal to 2 or 3.

\begin{lem}
The inflation map $H^i(H', K) \rightarrow H^i(H, K)$ is an isomorphism for all $i$.
\end{lem}
\begin{proof}
Since the order of $|G|$ ($2^{r+2}$) is invertible in $K$, we have $H^i(G, K) = 0$ for $i > 0$. There is an inflation-restriction exact sequence
$$
1 \rightarrow H^i(H', K) \rightarrow H^i(H,K) \rightarrow H^i(G, K)^{H'} \rightarrow \cdots
$$
By the vanishing of the cohomology of $G$, the claim follows. 
\end{proof}
Note that $\Gamma(2)$ contains a free subgroup $F_2$ on two generators \cite{Sa47}, generated by
$$
f_1 = \begin{pmatrix}
1 & 2 \\
0 & 1
\end{pmatrix}, \ f_2 = \begin{pmatrix}
1 & 0 \\
2 & 1
\end{pmatrix}.
$$
From now on, $F_2$ denotes this particular free subgroup.  In fact, $\Gamma(2) \cong \Z/2 \times F_2$, where $\Z/2$ is generated by $-I$.
\begin{lem}
Recall that $H$ lies in an extension
$$
\begin{tikzcd}
1 \ar[r] & A(2)^3 \ar[r]  & H \ar[r] & L \ar[r] & 1.
\end{tikzcd}
$$
The Lyndon-Hochschild-Serre spectral sequence associated to this exact sequence induces an isomorphism $H^2(H, K) \cong H^0(L, H^2(A(2)^3, K))$.
\end{lem}
\begin{proof}
The terms of this spectral sequence are
$$
E_2^{p,q} = H^p(L, H^q(A(2)^3, K)) \implies H^{p+q}(H,K).
$$
$L$ is a subgroup of $L_{max} = (\Z/2)^6 \rtimes S^3 \rtimes S_3$, which has order $2^{10} 3^{2}$. Therefore, $|L|$ is invertible in $K$. Furthermore, $H^q(A(2)^3, K)$ is a $K$ vector space, so we have the vanishing of $E_2^{p,q}$ for $p > 0$. Hence, there is an isomorphism
$$
H^2(H, K) \cong H^0(L, H^2(A(2)^3, K)) \cong H^2(A(2)^3, K)^L.
$$
\end{proof}

\begin{lem}
The restriction map $\rho: H^2(A(2)^3, K) \rightarrow H^2(F_2^3, K)$ is an isomorphism. Furthermore, let $p: L_{max} \rightarrow S^3 \rtimes S_3$ be the quotient by $(\Z/2)^6$. Then, there is an induced action of $L_0 := p(L)$ on $H^2(F_2^3, K)$, so that 
$$
H^2(A(2)^3, K)^{L} \cong H^2(F_2^3, K)^{L_0},
$$
under the previously mentioned isomorphism $\rho$.
\end{lem}
\begin{proof}
For the first claim, we know that $A(2)^3 \cong (\Z/2)^6 \times F_2^3$. Since $\text{Char} \ K \neq 2$, the restriction map $\rho$ is an isomorphism. Via this isomorphism, we can define an action of $L_{max}$ on $H^2(F_2^3, K)$. For the second claim, we must show that $(\Z/2)^6 \leq L_{max}$ acts trivally. Firstly, we use the K\"unneth formula to decompose $H^2(F_2^3, K)$.
\begin{equation}
H^2(F_2^3, K) \cong \bigoplus_{i+j+k=3} H^i(F_2, K) \otimes H^j(F_2, K) \otimes H^k(F_2, K)
\end{equation}
Note that $H^2(F_2, K) = 0$. Therefore, it suffices to show that $(\Z/2)^2$ acts trivially on $H^1(F_2, K)$. An element of the latter group is a homomorphism $\phi: F_2 \rightarrow K$. Now we unwind the definition of the action of $(\Z/2)^2 \leq L_{max}$ on $H^1(F_2, K) \cong \text{Hom}(F_2, K)$. Firstly, we extend $\phi$ trivially to a homomorphism $\widetilde{\phi}: A(2) \rightarrow K$ by defining it to vanish on $(\Z/2)^2 \leq A(2)$. Then, for $\epsilon \in (\Z/2)^2 \leq L_{max}$, we have $\epsilon \cdot \phi = \widetilde{\phi}(\epsilon \cdot (-) \cdot \epsilon^{-1})$. But $\widetilde{\phi}(\epsilon) = 0$ by construction, so the claim follows. Therefore, the action of $L$ on $H^2(F_2^3, K)$ factors through an action of $L_0$. 
\end{proof}

By the formula (3), we have $H^2(F_2^3, K) \cong K^{12}$. Our problem is now reduced to calculating the $L_0$ invariant submodule of $M = H^2(F_2^3, K)$. To do this, let us explain in detail how $L_0 \leq S^3 \rtimes S_3$ acts on $M$. 

Let us first work out the action of $S$ on $H^1(F_2, K)$. To fix notation, let $$
s = \begin{pmatrix}
0 & 1 \\
-1 & 0
\end{pmatrix}, 
t = \begin{pmatrix}
1 & 1 \\
0 & 1
\end{pmatrix}, 
r = \begin{pmatrix}
1 & 0 \\
1 & 1
\end{pmatrix}.
$$
So, $\bar{s},\bar{t},\bar{r} \in S$ generate the subgroups $B_3$, $B_1$, and $B_2$ respectively. We introduced before the elements $f_1$ and $f_2$ generating $F_2$. We now define a basis $\{ \hat{f}_1, \hat{f}_2 \}$ for $H^1(F_2, K)$, where $\hat{f}_{i}(f_j) = \delta_{ij} \in K$. So, we can work out the action of the elements $\bar{s},\bar{t},\bar{r}$ on $\hat{f}_i$. Firstly, it is easy to check that the following relations hold.
$$
\begin{matrix}
sf_1s^{-1} = f_2^{-1}, & tf_1t^{-1} = f_1, & rf_1r^{-1} = -f_2f_1^{-1}, \\
sf_2s^{-1} = f_1^{-1}, & tf_2t^{-1} = -f_1, f_2^{-1} & rf_2r^{-1} = f_2.
\end{matrix}
$$
From now on, we will abuse notation and refer to $\bar{s},\bar{t},\bar{r}$ as $s,t,r$. Then, using the equations above, we have the following.
$$
\begin{matrix}
s \cdot \hat{f}_1 = -\hat{f}_2 & t \cdot \hat{f}_1 = \hat{f}_1 & r \cdot \hat{f}_1 = \hat{f}_2 - \hat{f}_1\\
s \cdot \hat{f}_2 = -\hat{f}_1 & t \cdot \hat{f}_2 = \hat{f}_1 - \hat{f}_2 & r \cdot \hat{f}_2 = \hat{f}_2
\end{matrix}
$$
So, we have described the $S$ module structure on $H^1(F_2, K)$. $S$ acts trivially on $H^0(F_2, K) = K$. Then, each term $H^i(F_2, K) \otimes H^j(F_2, K) \otimes H^k(F_2, K)$ (which we denote by $H^{i,j,k}$) in the decomposition (3) becomes an $S^3$ module via the external tensor product. $H^2(F_2^3, K)$ is then the direct sum of these representations. Lastly, $S_3 \leq S^3 \rtimes S_3$ acts by permutation of the three summands in (3). Now we come to the main theorem of this section.

\begin{thm}
The group $H^2(H', K)$ is given in Table 2 in each of the ten cases with $h^{2,1}(X) = 3$.
\end{thm}
\begin{table}[ht]
    \centering
    \begin{tabular}{c|c}
        Case & $\text{dim}_K(H^2(H',K))$ \\
        \hline
        (0-1) & 0 \\
        (0-4) & $1$ \\
        \hline
        (1-1) & $1$ \\
        (1-5) & $1$ \\
        (1-11) & $2$ \\
        \hline
        (2-1) & $1$ \\
        (2-9) & $1$ \\
        (2-12) & $3$ \\
        \hline
        (3-5) & $1$ \\
        \hline
        (4-1) & $1$ \\
    \end{tabular}
    \caption{Cohomology of $H'$ with coefficients in $K$}
    \label{tab:my_label2}
\end{table}

\begin{proof}
By the previous lemmas, we are reduced to computing $H^2(F_2^3, K)^{L_0}$. The previous remark described the $L_0$ module structure on $H^2(F_2^3, K)$. We will break into cases, using the results of section 3.2 on the form of $L_0$. In all of the ten cases, $L_0$ is an extension of a group $T \leq S_3$ by a normal subgroup $N \leq S^3$. Therefore, to find the $L_0$ invariants, we will firstly determine the $N$ invariant submodule of $H^2(F_2^3, K)$, and subsequently the $T$ invariant submodule of the $H^2(F_2^3, K)^N$.\\ 

As we've said, $N$ acts on each term of the sum (3). A general element of $H^{1,1,0}$ has the form
$$
f = \sum_{i,j=1}^{2} \lambda_{ij} \hat{f}_i \otimes \hat{f}_j \otimes 1.
$$
We retain the $1$ in the expression above in order to distinguish the summands $H^{1,1,0}$, $H^{1,0,1}$, and $H^{0,1,1}$. For concreteness let us work with this factor $H^{1,1,0}$. We will write down some general results that will be used in each case. Namely, we determine the general form of such an element $f$ which is invariant under $(s,1,1)$, $(1,s,1)$, $(t,1,1)$, $(1,t,1)$, $(r,1,1)$, $(1,r,1)$, $(s,s,1)$, $(r,r,1)$, or $(t,t,1)$ respectively. In order, the elements $f$ are given by
$$
\lambda_{11} \hat{f}_1 \otimes \hat{f}_1 \otimes 1 + \lambda_{12} \hat{f}_1 \otimes \hat{f}_2 \otimes 1 - \lambda_{11} \hat{f}_2 \otimes \hat{f}_1 \otimes 1 - \lambda_{12} \hat{f}_2 \otimes \hat{f}_2 \otimes 1,
$$
$$
\lambda_{11} \hat{f}_1 \otimes \hat{f}_1 \otimes 1 - \lambda_{11} \hat{f}_1 \otimes \hat{f}_2 \otimes 1 + \lambda_{21} \hat{f}_2 \otimes \hat{f}_1 \otimes 1 - \lambda_{21} \hat{f}_2 \otimes \hat{f}_2 \otimes 1,
$$
$$
\lambda_{11} \hat{f}_1 \otimes \hat{f}_1 \otimes 1 +  \lambda_{12} \hat{f}_1 \otimes \hat{f}_2 \otimes 1,
$$
$$
\lambda_{11} \hat{f}_1 \otimes \hat{f}_1 \otimes 1 + \lambda_{21} \hat{f}_2 \otimes \hat{f}_1 \otimes 1, 
$$
$$
\lambda_{21} \hat{f}_2 \otimes \hat{f}_1 \otimes 1 + \lambda_{22} \hat{f}_2 \otimes \hat{f}_2 \otimes 1,
$$
$$
\lambda_{12} \hat{f}_1 \otimes \hat{f}_2 \otimes 1 +  \lambda_{22} \hat{f}_2 \otimes \hat{f}_2 \otimes 1,
$$
$$
\lambda_{11} \hat{f}_1 \otimes \hat{f}_1 \otimes 1 + \lambda_{12} \hat{f}_1 \otimes \hat{f}_2 \otimes 1 + \lambda_{12} \hat{f}_2 \otimes \hat{f}_1 \otimes 1 + \lambda_{11} \hat{f}_2 \otimes \hat{f}_2 \otimes 1,
$$
$$
\lambda_{11} \hat{f}_1 \otimes \hat{f}_1 \otimes 1 + \lambda_{12} \hat{f}_1 \otimes \hat{f}_2 \otimes 1 + \lambda_{12} \hat{f}_2 \otimes \hat{f}_1 \otimes 1 - 2\lambda_{12} \hat{f}_2 \otimes \hat{f}_2 \otimes 1,
$$
$$
-2\lambda_{12} \hat{f}_1 \otimes \hat{f}_1 \otimes 1 + \lambda_{12} \hat{f}_1 \otimes \hat{f}_2 \otimes 1 + \lambda_{12} \hat{f}_2 \otimes \hat{f}_1 \otimes 1 + \lambda_{22} \hat{f}_2 \otimes \hat{f}_2 \otimes 1.
$$
Checking these expressions is easy, if tedious. Now we are ready for a case by case analysis.\\

\textit{Case (0-1)}. $L_0 = S^3 \rtimes S_3$. We begin by finding the $S^3$ invariant elements of $H^{1,1,0}$. Imposing invariance under $(s,1,1)$ and $(1,s,1)$ implies that an invariant $f$ has the form
$$
f = \lambda ( \hat{f}_1 \otimes \hat{f}_1 \otimes 1 -\hat{f}_1 \otimes \hat{f}_2 \otimes 1 - \hat{f}_2 \otimes \hat{f}_1 \otimes 1 + \hat{f}_2 \otimes \hat{f}_2 \otimes 1).
$$
But then invariance under $(t,1,1)$ requires $\lambda_{21} = 0$, so $f = 0$. The same analysis applies to $H^{1,0,1}$ and $H^{0,1,1}$. Therefore, $H^2(A(2)^3,K)^L = 0$. \\

\textit{Cases (0-4), (1-1), (2-9)}. In each case, $L_0$ has the form $B_i^3 \rtimes S_3$. We may assume without loss of generality that $i = 3$ by relabelling basis elements, allowing us to consider the three cases at once. We begin by computing the $B_3^3$ invariant submodule of $H^{1,1,0}$. Imposing invariance under $(s,1,1)$ and $(1,s,1)$ yields that a general invariant $f$ has the form
$$
f = \lambda(\hat{f}_1 \otimes \hat{f}_1 \otimes 1 -\hat{f}_1 \otimes \hat{f}_2 \otimes 1 - \hat{f}_2 \otimes \hat{f}_1 \otimes 1 + \hat{f}_2 \otimes \hat{f}_2 \otimes 1),
$$
for any $\lambda \in K$. Applying the same analysis to $H^{1,0,1}$ and $H^{0,1,1}$ allows us to conclude that $B_3^3$ invariant submodule of $H^2(H', K)$ has the form $K \oplus K \oplus K$. Then $S_3$ acts to permute these factors (sending e.g. $\hat{f}_1 \otimes \hat{f}_1 \otimes 1$ to $\hat{f}_1 \otimes 1 \otimes \hat{f}_1$), so we conclude that $H^2(A(2)^3, K) = K$, the diagonal. \\

\textit{Case (2-12)}. $L_0 = (1 \times B_1^2) \rtimes S_3$. We begin by computuing the $1 \times B_1^2$ invariant submodule of $H^{1,1,0}$. We must impose invariance under $(1,t,1)$. A general invariant $f$ then has the form
$$
f = \lambda_{11} \hat{f}_1 \otimes \hat{f}_1 \otimes 1 + \lambda_{21} \hat{f}_2 \otimes \hat{f}_1 \otimes 1.
$$
Therefore the $1 \times B_1^2$ invariant submodule of $H^{1,1,0}$ has dimension 2. The same analysis applies to $H^{1,0,1}$. However, for $H^{0,1,1}$, we must impose invariance under both $(1,t,1)$ and $(1,1,t)$. Therefore a general invariant of $H^{0,1,1}$ has the form
$$
f = \lambda(1 \otimes \hat{f}_1 \otimes \hat{f}_1 - 1 \otimes \hat{f}_1 \otimes \hat{f}_2 - 1 \otimes \hat{f}_2 \otimes \hat{f}_1 + 1 \otimes \hat{f}_2 \otimes \hat{f}_2).
$$
So, the $1 \times B_1^2$ invariants of $H^{0,1,1}$ have dimension 1. Now, $S_2$ acts to permute the factors $H^{1,1,0}$ and $H^{1,0,1}$. Therefore, $H^2(A(2)^3,K)^{L} = K^2 \times K = K^3$. \\

\textit{Case (2-1)}. $L_0 = S \rtimes S_3$. Here $S \leq S^3$ is the diagonal. We start by finding the $S$ invariant submodule of $H^{1,1,0}$. We impose invariance under $(s,s,s)$, $(t,t,t)$, and $(r,r,r)$. A general invariant element $f$ must have the form
$$
f = \lambda(-2\hat{f}_1 \otimes \hat{f}_1 \otimes 1 + 1 \otimes \hat{f}_1 \otimes \hat{f}_2 \otimes 1 + \hat{f}_2 \otimes \hat{f}_1 \otimes 1 -  2 \hat{f}_2 \otimes \hat{f}_2 \otimes 1).
$$
The same analysis applies to finding the $S$ invariants of $H^{1,0,1}$ and $H^{0,1,1}$. They have the same form. So, the $S$ invariant submodule of $H^2(A(2)^3)$ has the form $K \oplus K \oplus K$, and $S_3$ permutes these factors. Therefore the $S_3$ invariants are the diagonal, and $H^2(A(2)^3, K)^L = K$.\\

\textit{Cases (1-5), (3-5)}. $L_0 = \widetilde{B}_i \rtimes S_3$. We may assume without loss of generality that $i = 3$. This case is actually identical to that of (0-4), (1-1), (2-9). Indeed, to find the $\widetilde{B}_3$ invariant submodule of $H^{1,1,0}$ we need to impose invariance under $(s,1,s)$ and $(1,s,s)$. Since these elements act in the same way as $(s,1,1)$ and $(1,s,1)$ on $H^{1,1,0}$, the $\widetilde{B}_3$ invariants are the same as the $B_3^2$ invariants. Proceeding with the same analysis as in those former cases, we find that $H^2(A(2)^3, K)^L = K$.\\

\textit{Case (1-11)}. $L_0 = (B_2^2 \times B_1) \rtimes S_2$. We start by finding the $B_2^2 \times B_1$ invariants of $H^{1,1,0}$. By imposing invariance under $(r,1,s)$ and $(1,r,s)$ we see that a general invariant element $f$ has the form
$$
f = \lambda \hat{f}_2 \otimes \hat{f}_2 \otimes 1.
$$
Now, we find that $B_2^2 \times B_1$ invariants of $H^{0,1,1}$. We impose invariance under $(1,t,1)$ and $(1,1,r)$. Invariance under $(1,t,1)$ yields $\lambda_{21} = \lambda_{22} = 0$. Invariance under $(1,1,r)$ yields $\lambda_{11} = 0$. So, a general invariant element has the form
$$
f = \lambda (1 \otimes \hat{f}_1 \otimes \hat{f}_2).
$$
The same analysis applies to $H^{1,0,1}$. So, the $B_2^2 \times B_1$ invariant submodule of $H^2(A(2)^3, K)$ has the form $K \oplus K \oplus K$. $S_2$ acts to permute the last two factors, so we conclude that $H^2(A(2)^3, K)^L \cong K^2$. \\

\textit{Case (4-1)}. In this case, $L_0$ is not a semidirect product, but we explained earlier that it lies in an extension of $S_3$ by a group $N$. $N$ is generated by the elements $(t,s,r)$, $(s,r,t)$, and $(r,t,s)$. We begin by finding the $N$ invariants in $H^{1,1,0}$. Imposing invariance under $(t,s,r)$ and $(s,r,t)$ forces an element $f$ to have the form 
$$
f = \lambda( \hat{f}_1 \otimes \hat{f}_1 \otimes 1 - 2\hat{f}_1 \otimes \hat{f}_2 \otimes 1 + \hat{f}_2 \otimes \hat{f}_1 \otimes 1 + \hat{f}_2 \otimes \hat{f}_2 \otimes 1).
$$
It is then easy to show that such an $f$ is already invariant under $(s,r,t)$. The same analysis applies to $H^{1,0,1}$ and $H^{0,1,1}$. Therefore the $N$ invariant submodule of $H^2(A(2)^3, K)$ has the form $K \oplus K \oplus K$, and $S_3$ acts to permute the three factors. Therefore, $H^2(A(2)^3, K)^L = K$. This concludes the proof of the claimed results.
\end{proof}

\begin{cor}
$H^2(H',\Z)$ has no $p$ torsion for $p > 3$.
\end{cor}
\begin{proof}
The group $H^2(H', \Z/p)$ has the form $(\Z/p)^a$, for some $a \in \mathbb{N}$. In contains a factor of $\Z/p$ for each (i) factor of $\Z$ in $H^2(H', \Z)$, (ii) factor of $\Z/p^k$ in $H^2(H', \Z)$ ($k \geq 1$) and (iii) each factor of $\Z/p^k$ in $H^3(H',\Z)$. The previous theorem establishes that $\text{dim}_{\Q}(H^2(H', \Q)) = \text{dim}_{\Z/p}(H^2(H', \Z/p))$ for $p > 3$. Hence, there are as many $\Z/p$ factors in $H^2(H', \Z/p)$ as there are factors of $\Z$ in $H^2(H',\Z)$. In particular, there is no $p$ torsion in $H^2(H', \Z)$. 
\end{proof}

\begin{cor}
The group $\text{Pic}(M_X)$, for $X$ such that $h^{2,1} = 3$, is isomorphic to $\Z^a \times A$, where $a$ is given in Table 2 above as $\text{dim}_K(H^2(H',K))$ for each case and $A$ is some (case dependent) finite abelian group of order $2^n3^m$, $n,m \in \mathbb{N}$. In particular, the Picard group of $M_X$ is infinite in 9 out of the 10 cases.
\end{cor}

\subsection{Hodge Bundles} In this section, we will give two proofs that the Hodge bundle over $M_X$ has finite order. The first proof is as follows.

\begin{thm}
Let $\lambda$ denote the Hodge bundle over $M_X$, for $h^{2,1}(X) = 3$. Then $\lambda$ is nontrivial, and has finite order. 
\end{thm}
\begin{proof}
To show that a line bundle $L$ over an analytic stack $\mathcal{X}$ has finite order, it suffices to exhibit a finite \'etale cover $\ell: \mathcal{Y} \rightarrow \mathcal{X}$ such that $\ell^*L$ has finite order. In our case, as we've discussed in the remark prior to Theorem 4.4, there is a finite \'etale cover $\ell: M(4)^3 \rightarrow M_X$. It is easy to see that that pullback $\ell^*\lambda$ is the Hodge bundle over $M(4)^3$. Since the Hodge bundle over $M(4)$ has finite order, the claim follows (and since $\ell^*\lambda$ is nontrivial, so must be $\lambda$). 
\end{proof}

For the second approach, we will explicitly construct a trivializing section of $\lambda^{\otimes 12}$. To do so, we must give a global nonvanishing section of $\pi^*\lambda$ over $\mathbb{H}^3$ which is $H'$ equivariant. To be precise, a line bundle over $[\mathbb{H}^3/H']$ is an $H'$ equivariant line bundle over $\mathbb{H}^3$, \textit{i.e.} a line bundle $L$ over $\mathbb{H}^3$ together with an $H'$ action. In these terms, the definition of the Hodge bundle $\lambda$ over $M_X$ is as follows. The associated line bundle over $\mathbb{H}^3$ is denoted $\pi^*\lambda$, and is by definition equal to the pushforward of the relative canonical bundle $\Omega^3_{\mathcal{U}^3/\mathbb{H}^3}$ to $\mathbb{H}^3$. The $H'$ action on this bundle is inherited from the $H'$ action on $\Omega^3_{\mathcal{U}^3}$, which is obtained by taking the differential of the action of an element $h \in H'$ on $\mathcal{U}^3$. Then to give a trivialization of $\lambda$ is to give a trivialization of $\pi^*\lambda$ which is equivariant with respect to the $H'$ actions on $\mathbb{H}^3$ and $\pi^*\lambda$.\\

So, we will seek a global nonvanishing section of $\pi^*\lambda$ of the form
$$
f(\tau_1,\tau_2,\tau_3) dz_1 \wedge dz_2 \wedge dz_3,
$$
for $f$ a nonvanishing analytic function on $\mathbb{H}^3$. Before we write down such a section, recall that the modular discriminant $\Delta: \mathbb{H}^3 \rightarrow \C$ is defined to be $\eta^{24}$, where $\eta$ is the Dedekind $\eta$ function
$$
\eta(\tau) = q^{\frac{1}{24}} \prod_{n=1}^{\infty} (1 - q^n),
$$
with $q = e^{2\pi i \tau}$ the square of the nome. Note that $\Delta$ is a modular form of weight 12.

\begin{thm} The section $\sigma: \mathbb{H}^3 \rightarrow (\pi^*\lambda)^{\otimes 12}$ defined by
$$
\sigma = \left(\prod_{i=1}^{3} \Delta(\tau_i)\right) (dz_1 \wedge dz_2 \wedge dz_3)^{\otimes 12}
$$
descends to a trivialization of $\lambda^{\otimes 12}$ over $M_X$.
\end{thm}
\begin{proof}
To check that $\sigma$ is equivariant under $H'$, it suffices to show that it is equivariant under $\Gamma^3$ and $S_3$ (since  the translations $(\Z/4)^6$ certainly act trivially on $\sigma$). The claim of equivariance under $S_3$ is immediate, since the only effect of a permutation is to introduce a sign coming from the wedge product $dz_1 \wedge dz_2 \wedge dz_3$. Equivariance under $\Gamma^3$ follows from the fact that $\Delta$ is modular of weight 12. Since $\Delta$ is also nonvanishing, this concludes the proof that $\sigma$ defines a trivializing section of $\lambda^{\otimes 12}$.
\end{proof}

We proceed to write down a globally defined K\"ahler potential on $M_X$. We consider the following function defined on $\mathbb{H}^3$
$$
\mathcal{K} = \log{\norm{\eta(\tau_1) \eta(\tau_2)\eta(\tau_3)}^2}.
$$
Clearly, $\mathcal{K}$ is invariant under the action of $S_3$. Furthermore, while $\eta$ is not invariant under $\Gamma$ (i.e. it is not modular), it satisfies $\eta(\gamma \tau) = \epsilon \eta(\tau)$, where $\gamma \in \Gamma$ and $\epsilon$ is some 12th root of unity. Therefore, $\mathcal{K}$ is invariant under the action of $H'$ on $\mathbb{H}^3$. Hence, it is well defined on the moduli space $M_X$. We can then define
$$
\omega = \frac{i}{2} \partial \bar{\partial} \mathcal{K}.
$$
which defines a (1,1) form on $M_X$. Since this form is positive definite, we see that $\omega$ defines a K\"ahler form on $M_X$. $\omega$ is in fact the curvature of the Hodge bundle. Indeed, we may define a $C^\infty$ section of $\lambda$ given by
$$
\widetilde{\sigma} = \norm{\eta(\tau_1) \eta(\tau_2)\eta(\tau_3)}^2 dz_1 \wedge dz_2 \wedge dz_3.
$$
This section is well defined for the same reason that $\mathcal{K}$ is well defined. If $h$ denotes the Weil-Petersson norm of $\widetilde{\sigma}$, then (after potentially rescaling $\widetilde{\sigma}$ by a constant) we have $\mathcal{K} = \log{h}$. Hence, $\omega$ is the curvature of $\lambda$ with respect to the Weil-Petersson metric.

\section{Open problems}

In general, it seems that very little is known about global Calabi-Yau moduli spaces.
Is their Picard group always finitely generated? 
If not, is the Hodge line bundle still of finite order?
Can the Hodge bundle ever be divisible?
Are the coarse moduli spaces always affine?

Clearly, it would be useful to have a global description of more examples.
The cases previously understood are 1-dimensional
(for example the mirror of the quintic threefold) 
or 2-dimensional. 
What about general toric hypersurfaces and complete intersections?
An obvious starting point might be the quintic threefold itself.
However, the large symmetry groups present (and the large dimension) 
may make this case particularly dificult.

At the opposite extreme, one might prefer to consider 
Calabi-Yau hypersurfaces in particularly ugly (or: random) toric varieties, 
ones whose only symmetries come from the torus action. 
In that case, one can hope to 
write down a normal form and get a global description
of the moduli space.  An example of such a normal form is in \cite{CDLW}.
These authors describe a two-dimensional moduli space of lattice-polarized K3s, 
which are compactifications of the Inose family.
These can in turn be described also as hypersurfaces in 
weighted projective space $W{\mathbb P}(5,6,22,33)$.

More generally, it may be possible to describe the global geometry
and Picard groups of the analogous moduli spaces for Borcea-Voisin
CYs, using the fact that moduli spaces of complex structures of
K3s and lattice-polarized K3s are locally homogeneous spaces.
In particular, one should be able to do this with the family in
\cite{CDLW} to get another three-dimensional Calabi-Yau moduli space.

A question closer to our actual results concerns the components of $M_{\overline{X}}$ 
and their Picard groups.  
We have seen that $M_{\overline{X}}$ has finitely generated Picard group; what about the non-central components? Likewise, is there a more direct way to compute $\text{Pic}\left(M_{\overline{X}} \right)$ that could obtain the 2 and 3 torsion subgroups?

\section{Acknowledgements}

We would like to thank P.~Aspinwall, B.~Conrad, C.~Doran,
R.~Hain, Z.~Komargodsky, D.~Morrison, R.~Plesser,
and N.~Seiberg for useful conversations, 
Patrick Vaudrevange for drawing our attention to the work \cite{FRTV13} and explaining it to us,
and Sheldon Katz and David Treumann for helpful comments on the manuscript.  
During the preparation of this work,
Ron Donagi was supported in part by 
NSF grant DMS 1603526 and by Simons Foundation grant \# 390287.
Eric Sharpe was supported in part by
NSF grant PHY-1417410.

\appendix

\section{Summary of results from \cite{DW} }

In this section, in Table 3, we summarize the results we need 
from  \cite{DW}, listing group actions and pertinent properties
of crepant resolutions of the resulting quotient.

As in \cite{DW},  
a group action labelled (r-n) refers to an action of the
group $({\mathbb Z}/2)^{r+2}$, where $G_T \cong ({\mathbb Z}/2)^2$
and $G_S \cong ({\mathbb Z}/2)^r$.  The number $n$ merely indexes
different actions of the same group.
We take the periods of the $i$-th elliptic curve to be $1, \tau_i$.
(In \cite{DW} the periods were doubled to $2, 2\tau_i$
in order to avoid halves in the half-periods.
In our current notation, the half-periods are $1/2, {\tau_i}/2$.)
A symbol such as $0 \pm$, $1 \pm$, or $\tau \pm$ denotes 
a reflection plus a translation by a half period on one elliptic factor.
Explicitly, $0\pm$ indicates that the generator acts as
\begin{displaymath}
z \: \mapsto \: \pm z,
\end{displaymath}
$1 \pm$ indicates that the generator acts as
\begin{displaymath}
z \: \mapsto \: \pm z + 1/2,
\end{displaymath}
and $\tau \pm$ indicates that the generator acts as
\begin{displaymath}
z \: \mapsto \: \pm z + \tau/2.
\end{displaymath}

An element of the twist group $G_T$ is denoted by a triple of such symbols 
(with an even number of negative signs.)
The entries in the shift subgroup $G_S$ are pure translations by (half of) the
indicated amount; we drop the unneeded $\pm$.  For example, $(\tau,\tau,0)$ indicates that the
generator acts as
\begin{displaymath}
z_1 \mapsto z_1 + \tau/2, \: \: \:
z_2 \mapsto z_2 + \tau/2, \: \: \:
z_3 \mapsto z_3.
\end{displaymath}

We also list the Hodge numbers of a crepant resolution, as well as 
the fundamental group $\pi_1$ of the same.  Possible fundamental groups are denoted as 
follows, in the same notation as \cite{DW}[Table 1]:
\begin{center}
\begin{tabular}{cl}
$A$: & the extension of ${\mathbb Z}/2$ by ${\mathbb Z}^2$, \\
$B$: & any extension of $({\mathbb Z}/2)^2$ by ${\mathbb Z}^6$, \\
$C$: & ${\mathbb Z}/2$, \\
$D$: & $({\mathbb Z}/2)^2$.
\end{tabular}
\end{center}

\vskip .5cm

The available information is displayed in two tables in \cite{DW}. A complete list, including 36 types, is given in Table 1. It is still not known exactly which pairs from that list may coincide. The possible coincidences were summarized in a second table (on page 17 of  \cite{DW}),  listing all undistinguished cases that may or may not turn out to coincide: this list consisted of seven pairs and one triple of items from Table 1 of  \cite{DW}. A subsequent work \cite{FRTV13} used a computer search to carefully analyze all those equivalences of distinct entries in Table 1 of  \cite{DW} that happen to be induced by affine linear transformations (twists and shifts) of the product of 3 tori. That work confirmed the completeness of the list in Table 1 of  \cite{DW}, and showed that precisely one pair of the previously undistinguished cases (items (3-1) vs (3-2)) coincided under such an equivalence. In the table below, we have therefore deleted entry (3-2). As far as we know, no progress has been made concerning the status of the other possible coincidences listed in the table on page 17 of  \cite{DW}: we have no information regarding the possible existence of exotic isomorphisms that are not induced by twists and shifts of the product of tori. We are grateful to Patrick Vaudrevange for drawing our attention to the work \cite{FRTV13} and explaining it to us.

\nopagebreak
\begin{table}[ht] 
\centering
\begin{tabular}{c | c c | c c}  
Label & $G_T$ & $G_S$ & $(h^{1,1}, h^{2,1})$ & $\pi_1$ \\ \hline
(0-1) & $(0+,0-,0-), (0-,0+,0-)$ & &  $(51,3)$ & $0$ \\
(0-2) & $(0+,0-,0-), (0-,0+,1-)$ & & $(19,19)$ & $0$ \\
(0-3) & $(0+,0-,0-), (0-,1+,1-)$ & & $(11,11)$ & $A$ \\
(0-4) & $(1+,0-,0-), (0-,1+,1-)$ & & $(3,3)$ & $B$ \\ \hline
(1-1) & $(0+,0-,0-), (0-,0+,0-)$ & $(\tau,\tau,\tau)$ & $(27,3)$ & $C$ \\
(1-2) & $(0+,0-,0-), (0-,0+,\tau-)$ & $(\tau,\tau,\tau)$ & $(15,15)$ & $0$ \\
(1-3) & $(0+,0-,0-), (0-,0+,1-)$ & $(\tau,\tau,\tau)$ & $(11,11)$ & $C$ \\
(1-4) & $(0+,0-,0-), (0-,1+,1-)$ & $(\tau,\tau,\tau)$ & $(7,7)$ & $A$ \\
(1-5) & $(1+,0-,0-), (0-,1+,1-)$ & $(\tau,\tau,\tau)$ & $(3,3)$ & $B$ \\
(1-6) & $(0+,0-,0-), (0-,0+,0-)$ & $(\tau,\tau,0)$ & $(31,7)$ & $0$ \\
(1-7) & $(0+,0-,0-), (0-,0+,1-)$ & $(\tau,\tau,0)$ & $(11,11)$ & $C$ \\
(1-8) & $(0+,0-,0-), (0-,1+,0-)$ & $(\tau,\tau,0)$ & $(15,15)$ & $0$ \\
(1-9) & $(0+,0-,0-), (0-,1+,1-)$ & $(\tau,\tau,0)$ & $(7,7)$ & $A$ \\
(1-10) & $(1+,0-,0-), (0-,1+,0-)$ & $(\tau,\tau,0)$ & $(11,11)$ & $A$ \\
(1-11) & $(1+,0-,0-), (0-,1+,1-)$ & $(\tau,\tau,0)$ & $(3,3)$ & $B$ \\ \hline
(2-1) & $(0+,0-,0-), (0-,0+,0-)$ & $(1,1,1), (\tau,\tau,\tau)$ &
$(15,3)$ & $D$ \\
(2-2) & $(0+,0-,0-), (0-,0+,1-)$ & $(1,1,1), (\tau,\tau,\tau)$ &
$(9,9)$ & $C$ \\
(2-3) & $(0+,0-,0-), (0-,0+,0-)$ & $(1,1,1), (\tau,\tau,0)$ & $(17,5)$ & $C$ \\
(2-4) & $(0+,0-,0-), (0-,0+,1-)$ & $(1,1,1), (\tau,\tau,0)$ & $(11,11)$ & $0$ \\
(2-5) & $(0+,0-,0-), (0-,0+,\tau-)$ & $(1,1,1), (\tau,\tau,0)$ & $(7,7)$ & $D$ \\
(2-6) & $(0+,0-,0-), (0-,0+,0-)$ & $(1,1,1), (\tau,1,0)$ & $(19,7)$ & $0$ \\
(2-7) & $(0+,0-,0-), (0-,0+,\tau-)$ & $(1,1,1), (\tau,1,0)$ & $(9,9)$ & $C$ \\
(2-8) & $(0+,0-,0-), (0-,\tau+,\tau-)$ & $(1,1,1), (\tau,1,0)$ & $(5,5)$ & $A$ \\
(2-9) & $(0+,0-,0-), (0-,0+,0-)$ & $(0,1,1), (1,0,1)$ & $(27,3)$ & $0$ \\
(2-10) & $(0+,0-,0-), (0-,0+,\tau-)$ & $(0,1,1), (1,0,1)$ & $(11,11)$ & $0$ \\
(2-11) & $(0+,0-,0-), (0-,\tau+,\tau-)$ & $(0,1,1), (1,0,1)$ & $(7,7)$ & $A$ \\
(2-12) & $(\tau+,0-,0-), (0-,\tau+,\tau-)$ & $(0,1,1), (1,0,1)$ & $(3,3)$ & $B$ \\
(2-13) & $(0+,0-,0-), (0-,0+,0-)$ & $(1,1,0), (\tau,\tau,0)$ & $(21,9)$ & $0$ \\
(2-14) & $(0+,0-,0-), (0-,0+,1-)$ & $(1,1,0), (\tau,\tau,0)$ & $(7,7)$ & $D$ \\
\hline
(3-1) & $(0+,0-,0-), (0-,0+,0-)$ & $(0,\tau,1), (\tau,1,0), (1,0,\tau)$ &
$(12,6)$ & $0$ \\
(3-3) & $(0+,0-,0-), (0-,0+,0-)$ & $(1,1,0), (\tau,\tau,0), (1,\tau,1)$ &
$(17,5)$ & $0$ \\
(3-4) & $(0+,0-,0-), (0-,0+,\tau-)$ & $(1,1,0), (\tau,\tau,0), (1,\tau,1)$ &
$(7,7)$ & $C$ \\
(3-5) & $(0+,0-,0-), (0-,0+,0-)$ & $(0,1,1), (1,0,1), (\tau,\tau,\tau)$ &
$(15,3)$ & $C$ \\
(3-6) & $(0+,0-,0-), (0-,0+,\tau-)$ & $(0,1,1), (1,0,1), (\tau,\tau,\tau)$ &
$(9,9)$ & $0$ \\ \hline
(4-1) & $(0+,0-,0-), (0-,0+,0-)$ & $(0,\tau,1), (\tau,1,0), (1,0,\tau),
(1,1,1)$ & $(15,3)$ & $0$
\end{tabular}
\vskip .5cm

\caption{Summary of Table 1 of \cite{DW}.}  \label{Table3}
\label{table:summary-dw}
\end{table}

\end{document}